\documentclass[12pt,a4paper,leqno]{article}

\usepackage[utf8]{inputenc}
\usepackage[T1]{fontenc}
\usepackage[english]{babel}
\usepackage{amsmath}
\usepackage{amsfonts}
\usepackage{amssymb}
\usepackage{multicol}
\usepackage{color}
\usepackage{soul}
\usepackage{dsfont}
\usepackage{pdfpages}
\usepackage{graphicx}
\usepackage{listings}
 \usepackage{hyperref}
\usepackage{fancyhdr}
\usepackage{algorithm}
\usepackage{comment}
\usepackage{mathabx}
\usepackage{algorithmic}

\title{Random field solutions to a linear SPDE driven by Lévy white noise}
\author{}
\date{}
\newcommand{\rc}{\right ]}
\newcommand{\lc}{\left [}
\newcommand{\R}{\mathbb{R}}

\newcommand{\lp}{\left(}
\newcommand{\rp}{\right)}
\newcommand{\N}{\mathbb{N}}

\newcommand{\E}{\mathbb{E}}

\newcommand{\Scd}{\mathcal S(\R^d)}
\newcommand{\I}{\leqslant}

\newcommand{\s}{\geqslant}
\newcommand{\vphi}{\varphi}

\newcommand{\PR}{\mathbb{P}}
\newcommand{\sL}{\mathcal{L}}
\newcommand{\sH}{\mathcal{H}}
\newcommand{\sO}{\mathcal{O}}
\newcommand{\sP}{\mathcal{P}}

\newcommand{\m}{\mathcal}

\newcommand{\dd}{\, \mathrm{d}}

\numberwithin{equation}{section}
\usepackage{ccaption}
\usepackage{amsthm}
\newtheorem{theo}{Theorem}[section]
\newtheorem{prop}[theo]{Proposition}
\newtheorem{rem}[theo]{Remark}
\newtheorem{df}[theo]{Definition}

\newcommand{\scal}[2]{ \langle #1 , #2  \rangle }

\newcommand{\limscd}[1]{\overset{\Scd}{\longrightarrow}}
\newcommand{\lims}[1]{\to}

\newcommand{\leb}[2]{\text{Leb}_{#1}\lp #2 \rp}
\newcommand{\nalpha}[1]{\dot #1^\alpha}
\allowdisplaybreaks[1]
\begin{document}
\begin{center}{\bf\Large Random field solutions to linear SPDEs driven
\vskip 12pt

  by symmetric pure jump Lévy space-time white noises
}
\vskip 16pt
{\bf Robert C.~Dalang}\footnote[1]{Institut de math\'ematiques,
Ecole Polytechnique F\'ed\'erale de Lausanne, Station 8,
CH-1015 Lausanne, Switzerland.
Emails: robert.dalang@epfl.ch, thomas.humeau@epfl.ch ~~

Partly supported by the Swiss National Foundation for Scientific Research.

This paper is based on a chapter in the Ph.D.~thesis of Th.~Humeau, written under the supervision of R.C.~Dalang.

 {\em MSC 2010 Subject Classifications:} Primary 60H15; Secondary 60G60, 60G51.
\vskip 12pt

 {\em Key words and phrases.} Linear stochastic partial differential equation, Lévy white noise, generalized stochastic process, random field solution, $\alpha$-stable noise.

}
and {\bf Thomas Humeau}$^1$
\vskip 16pt
Ecole Polytechnique F\'ed\'erale de Lausanne
\vskip 16pt

\end{center}

\begin{abstract}
We study the notions of mild solution and generalized solution to a linear stochastic partial differential equation driven by a pure jump symmetric Lévy white noise. We identify conditions for existence for these two kinds of solutions, and we identify conditions under which they are essentially equivalent. We establish a necessary condition for the existence of a random field solution to a linear SPDE, and we apply this result to the linear stochastic heat, wave and Poisson equations driven by a symmetric $\alpha$-stable noise.
 \end{abstract}

\section{Introduction}
In this article, we consider a linear stochastic partial differential equation of the form
\begin{equation}\label{equation}
\sL u= \dot X\, ,
\end{equation}
where $\sL$ is a partial differential operator and $\dot X$ is a symmetric pure jump Lévy white noise. We study two different notions of solution to \eqref{equation}. On the one hand, from the random field approach to SPDEs, we have the concept of mild solution, which is a random field defined as the convolution of a Green's function of $\sL$ with the noise. The mild solution is therefore defined as a stochastic integral, and some conditions are needed for its existence. For example, in the simple case of Gaussian white noise, the Green's function must be square integrable. The literature for the existence of mild solutions to SPDEs in the Gaussian case is already quite extensive (see \cite{minicourse, davar} for introductory lectures, and see \cite{dalang99, conusphd} for more advanced presentations). The case of Lévy noise has been less studied, but the existence of mild solutions for various equations has been considered in \cite{balan, chongheavytailed}, and the approach via evolution equations is considered in \cite{PZ}.

From the general theory of (deterministic) partial differential equations, we have the notion of weak solutions, or solutions in the sense of (Schwartz) distributions. Since the terms ``weak'' and ``distribution'' are often used with another meaning, we will instead use the term ``generalized solutions,'' in the spirit of the book \cite{gelfand}.

In this article, we are interested in the link between the notions of mild solution and generalized solution to the linear stochastic partial differential equation \eqref{equation}. More precisely, the questions that we study are the following:
\begin{itemize}
 \item[(1)] When it can be defined, is a mild solution also a generalized solution?
 \item[(2)] When a generalized solution exists,  under what conditions can it be represented by a random field?
 \item[(3)] What kinds of solution exist in the case of the stochastic heat equation, the stochastic wave equation, or the stochastic Poisson equation driven by an $\alpha$-stable noise?
\end{itemize}

A question related to (2) was studied in \cite[Theorem 11]{dalang99}. More precisely, this reference gives a necessary condition on the Green's function of the differential operator for the existence of a random field representation (see Definition \ref{defrfrepresentation}) for the generalized solution to an SPDE driven by a Gaussian colored noise.

To answer the above questions, we first introduce the two different notions of solution to a linear SPDE in Section \ref{notion_solution}. Then, in Section \ref{mild_is_generalized}, we provide an answer to question (1), first in the $\alpha$-stable case in Theorem \ref{mildisweak}, and then in a more general case in Theorem \ref{mild_is_weak_gen}. To prove these results, we also establish a new stochastic Fubini's theorem in Theorem \ref{generalfubini}, that is interesting in its own right. Section \ref{necessary_condition} deals with question (2), and a necessary condition for the generalized solution to have a random field representation is given in Theorem \ref{rfieldsolution} for the $\alpha$-stable case, and in Theorem \ref{rfieldsolution2} for a more general case. Finally, Section \ref{examples} deals with question (3), where we study applications of these results to the case of the stochastic heat equation, the stochastic wave equation and the stochastic Poisson equation in various dimensions. The main results can be found in Theorems \ref{heat_result}, \ref{wave_result} and \ref{poisson_result}.

\section{Notations and main definitions}
We will denote by $\m D(\R^d)$ the space of $C^\infty$ compactly supported functions, and $\m D'(\R^d)$ its topological dual space, the space of distributions or generalized functions (we refer the reader to \cite{distributions_schwartz} for an exposition of these notions). We suppose that $\sL$ is a partial differential operator with adjoint $\sL^*$ (think of $\sL$ as the heat or wave operator typically). We consider a fundamental solution $\rho\in \m D'(\R^{d})$ of the operator $\sL$, that is a solution to
\begin{equation}\label{rd1}
   \sL\rho=\delta_0 \qquad\mbox{in } \m D'(\R^{d}).
\end{equation}
The fundamental solution is not always unique (and choosing this solution typically amounts to imposing initial and/or boundary conditions), and in the following, we fix the choice of $\rho$. We recall the definition \cite{distributions_schwartz} of the convolution between a distribution $\rho$ and a smooth function with compact support $\varphi$:
\begin{equation}\label{rd2}
   \varphi *\rho (t):= \scal{\rho}{\varphi (t-\cdot)}\, .
\end{equation}
Note that this convolution defines a $C^\infty$ function. Also, for $\vphi \in \m D(\R^{d})$, we define $\scal{\check \rho} \vphi:=\scal \rho {\check \vphi}$, where for all $t\in \R^d$, $\check\vphi (t):=\vphi(-t)$. For any real valued function $f$, we will define $f_+:=\max(f,0)$.

Let $\dot X$ be a symmetric pure jump Lévy white noise on $S$, where $S$ is a Borel measurable subset of $\R^d$, with characteristic triplet $(0,0, \nu)$. More precisely, we suppose that there is a Poisson random measure $J$ on $S\times \R$ with intensity measure $\dd s \, \nu(\! \dd z)$ such that
$$X(\! \dd s):=\int_{|z| \I 1} z \tilde J(\! \dd s, \dd z)+\int_{|z| > 1} z  J(\! \dd s, \dd z)\, ,$$
and $\nu$ is a symmetric Lévy measure. As usual, $ \tilde J(\! \dd s, \dd z):=   J(\! \dd s, \dd z)-\dd s\, \nu(\! \dd z)$ is the compensated Poisson random measure associated to $J$. This Lévy white noise is a particular example of an independently scattered random measure introduced in \cite{rosinski}. For a link with other definition of Lévy white noise, we refer the reader to \cite{fageot_humeau}. In \cite[Theorem 2.7]{rosinski}, Rajput and Rosinski identified the space $L(\dot X, S)$ of deterministic functions that can be integrated with respect to such a noise. In particular, in our framework, we get that
$$L(\dot X, S)=\left \{ f: S\to \R \ \text{measurable} \ : \int_{S\times \R} \lp \left | f(s) z \right |^2 \wedge 1 \rp \dd s\, \nu(\! \dd z) <+\infty \right \}\, .$$
For properties of this space, we refer the reader to \cite[p. 466]{rosinski}. In particular, it is a linear complete metric space for a suitable norm, and the convergence $f_n \to f$ in $L(\dot X, S)$ as $n\to +\infty$ is equivalent to
$$\int_{S\times \R} \lp \left | \lp f_ n(s)-f(s)\rp z \right |^2 \wedge 1 \rp \dd s\, \nu(\! \dd z) \to 0 \, , \qquad \text{as} \ n \to +\infty\, .$$
In a spatio-temporal framework, we refer the reader to \cite{chong_integrability} for integrability conditions for non-deterministic integrands.

A generalized stochastic process (or generalized random field) $U$ is a linear map from a space of test functions $\m D(\R^d)$ into $L^0(\Omega)$ (the space of a.s.~finite random variables with the metric of convergence in probability). If, in addition, this map is continuous, then by \cite[Corollary 4.2]{spdewalsh}, it has a version $\tilde U$ (i.e. for any $\varphi \in \m D(\R^d)$, $\scal U \varphi = \langle \tilde U , \varphi \rangle$ a.s.) such that for almost all $\omega \in \Omega$, for any sequence $\varphi_n\to \varphi$ in $\m D(\R^d)$ as $n\to +\infty$, $\langle \tilde U , \varphi_n \rangle(\omega)\to \langle \tilde U , \varphi \rangle(\omega)$. That is, $\tilde U$ defines a random element in $\m D'(\R^d)$. In this case, $\tilde U$ is called a continuous generalized stochastic process, or a random distribution.


\section{Notions of solution to a linear SPDE}\label{notion_solution}
We introduce two different notions of solutions to the linear SPDE \eqref{equation} with associated fundamental solution $\rho$. Notice that in this framework, we are only considering the case where the Green's function of the operator $\sL$ is given by a shift of a fundamental solution.


\subsection{Generalized solution}

 In the following we will need a hypothesis on the fundamental solution $\rho$ of the differential operator $\sL$:
\begin{itemize}
 \item[\hypertarget{hyp1}{\textbf{(H1)}}]  $\rho$ is such that for any $\varphi \in \m D(\R^{d})$, the convolution $\varphi * \check \rho$ belongs to $L(\dot X, S)$.
\end{itemize}

The case where the noise is a symmetric $\alpha$-stable noise for some $\alpha \in (0,2)$ is already quite rich, and provides some insights into the general theory. More precisely, suppose that $\dot W^\alpha$ is an $\alpha$-stable symmetric Lévy white noise on $S$, with characteristic triplet $(0,0, \nu_\alpha)$, where $\nu_\alpha(\! \dd x)=\frac 1 {2|x|^{\alpha+1}} \dd x$. The characteristic function of $\dot W^\alpha$ is given by
\begin{equation*}
 \E\lp e^{iu \dot X(A)} \rp= \exp\left [-\leb d A |u|^\alpha \right ]\, , \qquad u\in \R\, ,
\end{equation*}
for any measurable set $A\subset S$ with finite Lebesgue measure. This notion coincides with that of a symmetric  $\alpha$-stable random measure developed in \cite[§3.3]{taqqu}. Since the skewness parameter $\beta$ vanishes, it is well known that a function $f:\R^d\to \R$ is $\dot W^\alpha$-integrable if and only if $f\in L^\alpha(S)$ (see \cite[§3.4]{taqqu}). In this framework, \hyperlink{hyp1}{\textbf{(H1)}} becomes
\begin{itemize}
 \item[\hypertarget{hyp1prime}{\textbf{(H1')}}] $\rho$ is such that for any $\varphi \in \m D(\R^{d})$, the convolution $\varphi * \check \rho$ belongs to $L^\alpha(S)$.
\end{itemize}

 We can then define a generalized solution to \eqref{equation}.
\begin{df}\label{gensol}
 Assume  \hyperlink{hyp1}{\textbf{(H1)}}. The \emph{generalized solution} to the stochastic partial differential equation \eqref{equation} is the linear functional $u_{\text{gen}}$ on $\m D(\R^{d})$ such that for all $\varphi \in \m D(\R^{d})$,
\begin{equation}\label{def1}
 \scal {u_{\text{gen}}} \varphi :=\scal{ \dot X}{\varphi * \check \rho}\, .
\end{equation}
\end{df}

\begin{rem}
 The generalized solution is in general not a distribution, since it may not define a continuous linear functional on $\m D(\R^{d})$. We may require additional properties on $\rho$ to have this property.
\end{rem}

\begin{rem}
 The functional $u_{\text{gen}}$ is a solution to \eqref{equation} in the weak sense: indeed, for $\varphi \in \m D(\R^{d})$,
\begin{align*}
 \scal{\sL u_{\text{gen}}}{\varphi}&= \scal{u_{\text{gen}}}{\sL^*\varphi}=\scal{\dot X}{\lp \sL^*\varphi\rp * \check\rho}\, .
\end{align*}
Also, by \eqref{rd1},
\begin{align*}
 \lp \sL^*\varphi\rp * \check \rho(t)=\scal{\check  \rho}{\sL^*\varphi (t-\cdot)}=\scal{\rho}{\sL^*\varphi (t+\cdot)}
 = \scal{\sL\rho}{\varphi (t+\cdot)}
 =\scal{\delta_0}{\varphi(t+\cdot)}=\varphi(t)\, ,
\end{align*}
Therefore, for all $\varphi \in \m D(\R^{d})$,
$$ \scal{\sL u_{\text{gen}}}{\varphi}=\scal{ \dot X}{\varphi} \, .$$
\end{rem}

A generalized solution cannot in general be evaluated pointwise. However, a generalized function (i.e. a distribution in the sense of Schwartz) can sometimes be represented by a true function. This is the motivation for the following definition.

\begin{df}\label{defrfrepresentation}
 We say a generalized stochastic process $u$ has a random field representation if there exists a jointly measurable random field $(Y_t)_{t\in \R^d}$ such that $Y$ has almost surely locally integrable sample paths, and for any $\varphi \in \m D(\R^d)$,
\begin{equation}\label{rfrepresentation}
 \scal u \varphi = \int_{\R^d} Y_t \, \varphi (t) \dd t\, \qquad \text{a.s.}
\end{equation}
\end{df}
The generalized stochastic processes that have a random field representation are exactly those which can be evaluated pointwise. For example, the Dirac distribution $\delta_0$ does not have a random field representation.


\subsection{Mild solution}
Generalized solutions are a useful generalization of classical solutions to a partial differential equation. However, non-linear operations on generalized functions are in general difficult to define, and we are often interested in finding solutions that can be evaluated pointwise. One type of solution that is often used in the SPDE literature is the notion of mild solution. Essentially, this consists in making use of the fundamental solution to write the equation in an integral form.
In order to be able to define a mild solution to \eqref{equation}, we will need another hypothesis on the fundamental solution $\rho$:

\begin{itemize}
 \item[\hypertarget{hyp2}{\textbf{(H2)}}]For any $t \in \R^d$, $\rho(t-\cdot) \in L(\dot X, S)$.
\end{itemize}

Again, in the case where the noise is a symmetric $\alpha$-stable noise for some $\alpha \in (0,2)$  \hyperlink{hyp2}{\textbf{(H2)}} becomes:
\begin{itemize}
 \item[\hypertarget{hyp2prime}{\textbf{(H2')}}]For any $t \in \R^d$, $\rho(t-\cdot) \in L^\alpha(S)$.
 \end{itemize}

\begin{df}\label{milddef}
Under hypothesis  \hyperlink{hyp2}{\textbf{(H2)}}, we define the \emph{mild solution} of \eqref{equation} via the formula
\begin{equation}\label{def}
 u_{\text{mild}}(t) := \scal{ \dot X}{\rho(t-\cdot)}\, .
\end{equation}
\end{df}
\begin{rem}
 When it exists, the mild solution is always a random field, while the generalized solution is defined as a distribution. It might turn out that the generalized solution has a random field representation, and we can then wonder if this representation is the mild solution. This question is investigated in Section \ref{necessary_condition}.
\end{rem}
The random field $u_{\text{mild}}$ defined in \eqref{def} has a jointly measurable version. This is a consequence of the following proposition.

\begin{prop}\label{measurableversion}
Let $f:\R^n\times \R^d \to \R$ be a Borel measurable function such that for any $t\in \R^n$, $f(t,\cdot)\in L(\dot X,S)$. For any $t\in \R^n$, let
$$u(t)=\scal{\dot X}{f(t,\cdot)}\, .$$
Then the random field $u$ has a jointly measurable version.
\end{prop}

\begin{proof}
See \cite[p. 926]{basse}.
\end{proof}


\section{When is a mild solution also a generalized solution?}\label{mild_is_generalized}
We point out that the generalized and mild solutions depend on the choice of the fundamental solution $\rho$. Therefore, once the choice of the fundamental solution has been made, it makes sense to study \emph{the} mild solution and \emph{the} generalized solution. For the remainder of this section, we fix the choice of a fundamental solution to the operator $\sL$.

The generalized solutions $u_{\text{gen}}$ and the mild solution $u_{\text{mild}}$ (under  \hyperlink{hyp1}{\textbf{(H1)}} and  \hyperlink{hyp2}{\textbf{(H2)}}, respectively) are defined by in \eqref{def1} and \eqref{def}.
Therefore, in general, if $u_{\text{mild}}$ has locally integrable sample paths, then for any $\varphi \in \m D(\R^d)$,
$$\scal{u_{\text{mild}}}{\varphi}:= \int_{\R^d} u_{\text{mild}}(t) \varphi(t) \dd t=\int_{\R^d} \scal{\dot X}{\rho(t-\cdot)}\varphi(t) \dd t\, .$$
We see, in particular, that if we can exchange the stochastic integral and the Lebesgue integral, then we get
$$\scal{u_{\text{mild}}}{\varphi}= \scal{\dot X}{\int_{\R^d} \rho(t-\cdot)\varphi(t) \dd t }=\scal{\dot X}{\varphi* \check \rho}=\scal{u_{\text{gen}}}{\varphi}\, .$$
Therefore, provided the exchange of the order of integration is valid, we have $u_{\text{mild}}=u_{\text{gen}}$ in the sense of generalized stochastic processes, and in order to answer the question of when the mild solution is also the generalized solution, we need a stochastic Fubini's theorem.


\subsection{The $\alpha$-stable case}
We first consider the case of an $\alpha$-stable symmetric noise, in which a very complete result can be given.

\begin{theo}\label{mildisweak}
 Assume  \hyperlink{hyp2prime}{\textbf{(H2')}}. Let $u_{\text{mild}}$ be a jointly measurable version of the mild solution to \eqref{equation} defined in \eqref{def}, where the noise is symmetric and $\alpha$-stable. For any $\varphi \in \m D(\R^d)$, let $\mu_\varphi(\! \dd t)= |\varphi (t) | \dd t$.
 \begin{itemize}
\item[(i)] If $\alpha > 1$, and for any $\varphi \in \m D(\R^d)$,
\end{itemize}
\vspace{-1em}
\begin{equation}\label{alpha0}
 \int_{\R^d}\lp \int_{S} \left | \rho(t-s) \right |^\alpha \dd s \rp^{\frac 1 \alpha}  \mu_\varphi(\! \dd t) <+\infty \, ,
\end{equation}
\begin{itemize}
\item[] then $u_{\text{mild}}=u_{\text{gen}}$ in the sense of generalized stochastic processes.
\item[(ii)] If $\alpha =1$, and $\rho$ is such that for any $\varphi \in \m D(\R^d)$,
\end{itemize}
\vspace{-1em}
\begin{equation}\label{alpha1}
 \int_{\R^d}\mu_\varphi(\! \dd t)  \int_{S} \! \dd s \, |  \rho(t-s)| \left [ 1+ \log_+ \lp \frac{| \rho(t-s)| \int_{\R^d} \mu_\varphi(\!\dd r)  \int_S \dd v | \rho(r-v)|    }{\lp\int_S | \rho(t-v)| \dd v \rp\lp \int_{\R^d} | \rho( r-s) | \mu_\varphi( \!\dd r)\rp} \rp \right ] <+\infty\, .
\end{equation}
\begin{itemize}
\item[]then $u_{\text{mild}}=u_{\text{gen}}$ in the sense of generalized stochastic processes.
\item[(iii)] If $\alpha <1$, and $\rho$ is such that for any $\varphi \in \m D(\R^d)$,
\end{itemize}
\vspace{-1em}
\begin{equation}\label{alphag1}
 \int_{S} \lp \int_{\R^d} |\rho(t-s)| \mu_\varphi (\! \dd t) \rp^\alpha \dd s <+\infty\, ,
\end{equation}
\begin{itemize}
\item[]then $u_{\text{mild}}=u_{\text{gen}}$ in the sense of generalized stochastic processes.
\end{itemize}
\end{theo}

\begin{rem}\label{mildisweakrem}
It is not difficult to check the following statements:
\begin{itemize}
 \item[(1)]  Condition \eqref{alpha0} is equivalent to:
 \end{itemize}
 \vspace{-1em}
\begin{equation*}
 t\mapsto \|\rho(t-\cdot)\|_{L^\alpha(S)} \in L^1_{\text{loc}}(\R^d)\, .
\end{equation*}
\begin{itemize}
\item[]
\begin{itemize}
 \item[(a)] If $S=\R_+^d$, and $\rho(t)=0$ for all $t\in \R^d\setminus \R_+^d$, then \eqref{alpha0} is equivalent to $\rho \in L^\alpha_{loc}(\R_+^d)$.
 \item[(b)] If $S=\R_+\times \R^{d-1}$, and $\rho(t,x)=0$ if $t<0$, then \eqref{alpha0} is equivalent to
 \end{itemize}
 \end{itemize}
 \vspace{-1em}
\begin{equation*}
 \forall t>0, \ \int_0^t\int_{\R^{d-1}} |\rho(s,y)|^\alpha \dd s \dd y <+\infty\, .
\end{equation*}
\begin{itemize}
\item[(2)]
Similarly, condition \eqref{alphag1} is equvivalent to:
\end{itemize}
\vspace{-1em}
\begin{equation*}
 \text{for any compact } K\subset \R^d, \int_S \lp \int_K |\rho(t-s) | \dd t \rp ^\alpha \dd s <+\infty\, .
\end{equation*}
\begin{itemize}
\item[]
\begin{itemize}
 \item[(a)] If $S=\R_+^d$, and $\rho(t)=0$ for all $t\in \R^d\setminus \R_+^d$, then \eqref{alphag1} is equivalent to $\rho \in L^1_{loc}(\R_+^d)$.
 \end{itemize}
 \end{itemize}

\end{rem}

\begin{proof}[Proof of Theorem \ref{mildisweak}.]
We begin with \textit{(i)}. As mentioned above, we need a stochastic Fubini theorem to exchange the Lebesgue integral and the stochastic integral. Since $\varphi \in \m D(\R^d)$, the measure $\mu_\varphi$ is finite. By \eqref{alpha0} and \cite[Theorem 11.3.2]{taqqu}, $\int_{\R^d} |u_{\text{mild}}(t)| \mu_{\varphi}(\!\dd t) <+\infty$ a.s (that is, the sample paths of $u_{\text{mild}}$ are almost surely locally integrable, and $u_{\text{mild}}$ defines a generalized random process). By the stochastic Fubini Theorem in \cite[Theorem 11.4.1]{taqqu},
\begin{equation*}
 \int_{\R^d} u_{\text{mild}}(t)  \varphi (t) \dd t = \int_{S} \lp \int_{\R^d} \rho(t-s) \varphi (t) \dd t \rp \dot W^\alpha (\!\dd s)=\scal{\dot W_\alpha}{\varphi* \check \rho}\, .
\end{equation*}
Therefore, for any $\varphi \in \m D(\R^d)$,
$$\scal{u_{\text{mild}}}{\varphi} = \scal{\dot W_\alpha}{\varphi* \check \rho}=: \scal{u_{\text{gen}}}{\varphi}\, ,$$
and therefore $u_{\text{mild}}=u_{\text{gen}}$ in the sense of generalized stochastic processes.

The proof of \textit{(ii)} and \textit{(iii)} follows the same steps, with the difference that the conditions \eqref{alpha1} and \eqref{alphag1} are necessary to apply \cite[Theorem 11.3.2]{taqqu} when $\alpha=1$ or $\alpha>1$.
\end{proof}

The careful reader may wonder if  \hyperlink{hyp1prime}{\textbf{(H1')}} is satisfied in these cases, since it is a necessary condition for the existence of the generalized solution. In fact, \eqref{alpha1} and \eqref{alphag1} immediately imply Hypothesis \hyperlink{hyp1prime}{\textbf{(H1')}} when $\alpha\I 1$, and by Minkowski's inequality for integrals \cite[A.1]{stein}, \eqref{alpha0} also implies  \hyperlink{hyp1prime}{\textbf{(H1')}} when $\alpha>1$.


\subsection{A Stochastic Fubini Theorem}
In this section, we suppose that the driving noise $\dot X$ is a pure jump symmetric Lévy white noise, that is, a Lévy white noise with characteristic triplet $(0,0, \nu)$, where $\nu$ is a symmetric Lévy measure. We can no longer rely on the pre-existing work on $\alpha$-stable random measures exposed in \cite{taqqu}, and we need another version of a stochastic Fubini theorem. For convenience, we provide here a Fubini's theorem for integrals with respect to this Lévy noise. Such stochastic Fubini theorems for $L^0$-valued random measures already exist in the literature. For instance, \cite[Corollary 1]{lebedev} is more general (it deals with stochastic integrands), but integration of non-deterministic processes with respect to Lévy white noises relies on a space-time framework, where the time component is critical for the definition of predictable processes.
\begin{theo}\label{generalfubini}
Let $\dot X$ be a symmetric pure jump Lévy white noise on $S\subset \R^d$, with characteristic triplet $(0, 0, \nu)$ and jump measure $J$. Let $f:S\times \R^n \mapsto \R$ be measurable and such that for any $t\in \R^n$, $f(\cdot ,t) \in L(\dot X, S)$, and let $\mu$ be a finite measure on $\R^n$. Suppose that
\begin{equation}\label{l1}
 \int_{\R^n} \left | \scal {\dot X}{f(\cdot, t)}\right | \mu(\! \dd t)<+\infty \, , \qquad \text{a.s.}
\end{equation}
Then, for almost all $s\in S$, $f(s,\cdot) \in L^1(\mu)$, and the function $\mu \circledast f: s \mapsto \int_{\R^n} f(s,t) \mu (\! \dd t)$ is in $L(\dot X, S)$, and
\begin{equation}\label{eqfubini}
 \int_{\R^n}  \scal {\dot X}{f(\cdot,t)} \mu(\! \dd t) = \scal{\dot X}{\mu \circledast f}\, \qquad \text{a.s.}
\end{equation}
\end{theo}

\begin{rem}
We emphasize that the $\circledast$ operation is not commutative. In particular, it involves a measure and a measurable function whose roles are not interchangeable.
\end{rem}

\begin{proof}[Proof of Theorem \ref{generalfubini}]
The main probability space is $\lp \Omega, \m F, \PR \rp$. Since $\mu$ is a finite measure, we can suppose without loss of generality that it is a probability measure on $\R^n$. Let $\lp \Omega', \mathcal F', \PR'\rp$ be a probability space, and $(T_i)_{i\s1}$ be a sequence of i.i.d. random variables on this space with law $\mu$. We write $\E'$ for the expectation with respect to the probability measure $\PR'$. In this framework, \eqref{l1} is equivalent to
\begin{equation*}
 \E'\lp \left | \scal{\dot X}{ f(\cdot , T_1)}\right | \rp <+\infty\, \qquad \PR-\text{a.s.}
\end{equation*}
(we are using the jointly measurable version of $\scal {\dot X}{f(\cdot, t)}$ provided by Proposition \ref{measurableversion}). More precisely, there is a set $\Omega_1\subset \Omega$ such that $\PR(\Omega_1)=1$, and for any $\omega \in \Omega_1$,
\begin{equation*}
 \E'\lp \left | \scal{\dot X}{ f(\cdot , T_1)}(\omega)\right | \rp <+\infty\, .
\end{equation*}
By the strong law of large numbers, for any $\omega\in \Omega_1$, there is a set $\Omega_1'(\omega)\subset \Omega'$ such that $\PR'\lp \Omega_1'(\omega)\rp=1$ and for any $\omega'\in \Omega_1'(\omega)$,
\begin{equation}
 \label{slln}
 \frac 1 n \sum_{i=1}^n  \scal{\dot X}{ f(\cdot , T_i(\omega'))}(\omega) \to  \E'\lp  \scal{\dot X}{ f(\cdot , T_1)} (\omega) \rp \qquad \text{as} \ n\to +\infty\, .
\end{equation}
We define
$$A=\left\{ (\omega, \omega')\in \Omega\times \Omega' : \eqref{slln}\ \text{occurs}\right \}\, .$$
Then $A \in \m F \times \mathcal F'$. For $\omega\in \Omega$, let
$$A_\omega=\left \{ \omega' \in \Omega' : (\omega, \omega')\in A\right \}\, .$$
Then, for any $\omega \in \Omega_1$, $\PR'\lp A_\omega\rp =1$, and we deduce that $\PR\times \PR'(A)=1$.

For any $n\in \N, s\in S$ and $\omega'\in \Omega'$, we set $f_n(s,\omega')=\frac 1 n \sum_{i=1}^n f(s, T_i(\omega'))$. Then, $f_n(\cdot, \omega') \in L(\dot X, S)$ since this is a vector space. For any $\omega'\in \Omega'$, there is a set $\Omega_n(\omega') \subset \Omega$ such that $\PR\lp \Omega_n(\omega')\rp=1$ and for any $\omega\in \Omega_n(\omega')$,
\begin{equation}\label{linearity}
  \frac 1 n \sum_{i=1}^n  \scal{\dot X}{ f(\cdot , T_i(\omega'))}(\omega)=\scal{\dot X}{f_n(\cdot, \omega')}(\omega)\, .
\end{equation}
For any $\omega'\in \Omega'$, the set $\Omega_\infty (\omega')=\bigcap_{n=1}^{+\infty}\Omega_n(\omega')$ is such that $\PR\lp \Omega_\infty(\omega')\rp=1$ and for any $\omega\in \Omega_\infty(\omega')$, \eqref{linearity} holds for all $n\in \N$. We define
$$B=\left\{ (\omega, \omega')\in \Omega\times \Omega' : \eqref{linearity}\ \text{occurs for all }n\in \N \right \}\, .$$
Then, for $\omega'\in \Omega'$, let
$$B^{\omega'}=\left \{ \omega \in \Omega : (\omega, \omega')\in B\right \}\, .$$
For any $\omega' \in \Omega'$, $\PR\lp B^{\omega'}\rp =1$, and we deduce that
$$\int_{\Omega'}\PR\lp B^{\omega'}\rp \PR'(\! \dd \omega ')  =1\, .$$
By Fubini's theorem, we deduce that
\begin{equation}\label{rd3}
   \int_{\Omega'}\lp \int_\Omega \mathds 1_{(\omega, \omega')\in A\cap B} \, \PR(\! \dd \omega) \rp \PR'(\! \dd \omega ') =\int_{\Omega}\lp \int_{\Omega'} \mathds 1_{(\omega, \omega')\in A\cap B} \,\PR'(\! \dd \omega') \rp \PR(\! \dd \omega )  =1\, .
\end{equation}
Let $\omega'\in \Omega$. We define
$$\lp A\cap B\rp^{\omega'}=\left \{ \omega \in \Omega : (\omega, \omega') \in A\cap B \right\}\, .$$
From \eqref{rd3}, for $\PR'$-almost all $\omega' \in \Omega'$, $\PR\lp \lp A\cap B\rp^{\omega'}\rp=1$. In other words, for $\PR'$-almost all $\omega'\in \Omega'$,
$$ \frac 1 n \sum_{i=1}^n  \scal{\dot X}{f(\cdot ,T_i(\omega'))}(\omega) =\scal{\dot X}{f_n(\cdot, \omega')}(\omega) \to  \E'\lp  \scal{\dot X}{f(\cdot , T_1)} (\omega) \rp \qquad \text{as} \ n\to +\infty\, ,$$
for $\PR$-almost all $\omega \in \Omega$. In particular, for $\PR'$-almost all $\omega'\in \Omega$, the sequence of random variables $( \scal{\dot X}{f_n(\cdot, \omega')} )_{n\s 1}$ on $\lp \Omega , \mathcal F , \PR\rp$ is a Cauchy sequence in probability. By $\PR$-a.s. linearity of $\dot X$ and the isomorphism property in \cite[Theorem 3.4]{rosinski}, which uses the symmetry of $\dot X$ (see \cite[Proposition 3.6]{rosinski}, we deduce that $\lp f_n(\cdot, \omega')\rp_{n\s 1}$ is a Cauchy sequence in $L(\dot X, S)$. By completeness, for $\PR'$-almost all $\omega' \in \Omega'$, there is a function $\tilde f(\cdot, \omega')\in L(\dot X, S)$ such that $f_n(\cdot, \omega') \to \tilde f(\cdot , \omega')$ as $n \to+\infty$ in $L(\dot X, S)$ (see \cite{rosinski} for the definition of that convergence, in particular, it implies the convergence in measure on compact subsets of $S$ \cite[p.466]{rosinski}). By \eqref{l1} and \cite[Theorem 6]{path} (which also uses the symmetry of $\dot X$), for almost every $s\in S$, $\int_{\R^d} | f(s, t) | \mu(\! \dd t)<+\infty$, that is $\E'\lp |  f(s, T_1) |\rp<+\infty$. By the strong law of large numbers, we deduce that for almost all $s\in S$, there is a set $\Omega'_s$ such that $\PR'(\Omega'_s)=1$ and for any $\omega'\in \Omega'_s$,
\begin{equation}
 \label{cv}
\frac 1 n \sum_{i=1}^n f(s, T_i(\omega')) = f_n(s,\omega') \to \E'\lp  f(s, T_1) \rp= \mu \circledast f(s) \qquad \text{as} \ n \to +\infty .
\end{equation}
Let $C=\left \{ (s,\omega') \in S\times \Omega' : \eqref{cv} \ \text{holds} \right \}$, for $s\in S$,  $C_s=\left \{\omega'\in \Omega' : (s, \omega') \in C\right \}$, and for $\omega'\in \Omega'$, $C^{\omega'}=\left \{s\in S : (s, \omega') \in C\right \}$ . Since for almost all $s\in S$, $\PR'\lp C_s\rp=1$, by Fubini's theorem,
we deduce that for almost all $\omega'\in \Omega'$, \eqref{cv} holds for almost every $s\in S$ (with respect to Lebesgue measure). We can then drop the dependence in $\omega'$, so that there is a sequence $(t_i)_{i\s1}$ of deterministic times (for $\PR$) in $\R^n$ such that
\begin{equation}\label{aslim}
 \frac 1 n \sum_{i=1}^n f(s, t_i) \to \mu \circledast f(s) \qquad \text{a.e. in} \ s \ \text{as} \ n \to +\infty ,
 \end{equation}
 \begin{equation}\label{truc}
  \frac 1 n \sum_{i=1}^n  \scal{\dot X}{f(\cdot , t_i)} =\scal{\dot X}{\frac 1 n \sum_{i=1}^n f(\cdot , t_i)} \to  \int_{\R^d}  \scal{\dot X}{f(\cdot , t)} \mu(\! \dd t)  \qquad \PR-\text{a.s.},
\end{equation}
as $n\to+\infty$, and
\begin{equation}\label{llim}
 \frac 1 n \sum_{i=1}^n f(\cdot, t_i) \to \tilde f(\cdot) \qquad \text{in } L(\dot X, S) \ \text{as} \ n \to +\infty .
\end{equation}
Since convergence in $L( \dot X , S)$ implies convergence almost everywhere along a subsequence (see \cite[p. 466]{rosinski}), by uniqueness of the limit we get from \eqref{aslim} and \eqref{llim} that $\mu\circledast f=\tilde f$ almost everywhere (and hence $\tilde f$ does not depend on $\omega'$), and $\frac 1 n \sum_{i=1}^n f(\cdot, t_i) \to \mu\circledast f$ in $L(\dot X, S)$. Therefore,
\begin{equation}\label{truc22}
 \scal{\dot X}{\frac 1 n \sum_{i=1}^n f(\cdot,t_i)} \to \scal{\dot X}{\mu \circledast f} \qquad \text{as} \ n\to +\infty\, ,
\end{equation}
in $\PR$-probability. By uniqueness of the limit, gathering \eqref{truc} and \eqref{truc22}, we deduce that $\PR$-almost surely, \eqref{eqfubini} holds.
\end{proof}


\subsection{The general case}

In this section, we suppose again that the driving noise $\dot X$ is a pure jump symmetric Lévy white noise, that is, a Lévy white noise with characteristic triplet $(0,0, \nu)$, where $\nu$ is a symmetric Lévy measure. We can now apply Theorem \ref{generalfubini} to our problem.

\begin{theo}\label{mild_is_weak_gen}
 Assume  \hyperlink{hyp2}{\textbf{(H2)}}. Let $u_{\text{mild}}$ be a measurable version of the mild solution to \eqref{equation} defined in \eqref{def}. Suppose that the sample paths of $u_{\text{mild}}$ are almost surely locally integrable with respect to Lebesgue measure. Then $u_{\text{mild}}=u_{\text{gen}}$ in the sense of generalized stochastic processes.
\end{theo}

\begin{proof}
Let $\varphi \in \m D(\R^d)$, and let $\mu_{\varphi}^+(\! \dd t):=\varphi_+(t) \dd t $ and $\mu_{\varphi}^-(\! \dd t):=\varphi_-(t) \dd t $, where $\varphi_+=\max(\varphi, 0)$ and $\varphi_-=max(-\varphi, 0)$ are, respectively, the positive and negative parts of $\varphi$. These two measures are finite, and are the positive and negative parts of the signed measure $\mu_\varphi(\!\dd t):= \varphi(t) \dd t$. Since $u_{\text{mild}}$ has almost surely locally integrable sample paths,
\begin{equation*}
 \int_{\R^d} \left | u_{\text{mild}}(t)\right | \mu_\varphi^\pm(\! \dd t) <+\infty\, .
\end{equation*}
Therefore, we can apply Theorem \ref{generalfubini} separately with the positive and negative part of $\mu_\varphi$, and recombining them and using \eqref{def} and \eqref{def1} yields
\begin{equation*}
 \scal{u_{\text{mild}}}{\varphi}:=\int_{\R^d} u_{\text{mild}}(t)\varphi(t) \dd t=\int_{\R^d} \scal{\dot X}{\rho(t-\cdot)} \mu_\varphi(\! \dd t)=\scal{\dot X}{\varphi* \check \rho}=\scal{u_{\text{gen}}}{\varphi}\, ,
\end{equation*}
which proves the claim.
\end{proof}

\begin{rem}
(a) In the $\alpha$-stable case, we had a necessary and sufficient condition for the sample paths of the mild solution to be locally integrable. In the general case, we do not have such precise statement, we only have the necessary condition of \cite[Theorem 6]{path}.

(b) Again, one might wonder if Hypothesis \hyperlink{hyp1}{\textbf{(H1)}} is satisfied, and it turns out that $\varphi*\check \rho=\mu_\varphi \circledast  f$, where $f(s,t):=\rho(t-s)$, and by Theorem \ref{generalfubini},  $\mu_\varphi \circledast f$ is $\dot X$-integrable, so \hyperlink{hyp1}{\textbf{(H1)}} is satisfied and the generalized solution is well defined.

(c) Having almost surely locally integrable sample paths is the minimum requirement for a stochastic process to be considered as a generalized stochastic process, since we need to be able to integrate it against any test function. Essentially, Theorem \ref{mild_is_weak_gen} states that if the mild solution can be considered as a generalized stochastic process, then it must be equal to the generalized solution.

\end{rem}



\section{Necessary condition for the existence of a random field solution}\label{necessary_condition}


In this section, we aim to answer the following question. Suppose that \hyperlink{hyp1}{\textbf{(H1)}} is satisfied. Then, the generalized solution can be defined as in Definition \ref{def1}. Suppose also that the generalized solution has a random field representation $Y$. Then, is  \hyperlink{hyp2}{\textbf{(H2)}} satisfied? And if so, is $Y$ the mild solution?

In the case of a Gaussian noise, that can be spatially correlated, this question has already been investigated under slightly different assumptions in \cite[Theorem 11]{dalang99}. Transposed to our framework, this theorem implies that in the case of an SPDE driven by Gaussian white noise (in space and time), if the generalized solution has a random field representation, then the fundamental solution of this SPDE is necessarily square integrable. Here, we extend this kind of statement to the setting of symmetric pure jump Lévy white noises.

\subsection{The $\alpha$-stable case}

Again, we first restrict to the case of a symmetric $\alpha$-stable noise, for some $\alpha \in (0,2)$, where we have a very complete result.
\begin{theo}\label{rfieldsolution}
Assume  \hyperlink{hyp1prime}{\textbf{(H1')}}. Let $u_{\text{gen}}$ be the generalized solution to \eqref{equation} defined by \eqref{def1}. Suppose that $u_{\text{gen}}$ has a random field representation $Y$ in the sense of Definition \ref{defrfrepresentation}, that is there exists a jointly measurable random field $(Y_t)_{t\in \R^d}$ such that $Y$ has almost surely locally integrable sample paths, and for any $\varphi\in \m D(\R^d)$,
\begin{equation}\label{eqdef}
 \scal{u_{\text{gen}}}{\varphi}= \int_{\R^d} Y_t \varphi(t) \dd t \qquad \text{a.s.}
\end{equation}
Then, for almost all $t\in \R^d$, $\rho(t-\cdot) \in L^\alpha(S)$ (i.e.  \hyperlink{hyp2prime}{\textbf{(H2')}} is satisfied almost everywhere), and
\begin{equation}\label{mildresult}
 Y_t =\scal{\dot W^\alpha}{\rho(t-\cdot)}=u_{\text{mild}}(t) \qquad \text{a.s.} \quad \text{a.e.}
\end{equation}
 Furthermore, for any $\psi \in \m D(\R^d)$, conditions \eqref{alpha0}--\eqref{alphag1} of Theorem \ref{mildisweak}, for $\alpha >1$, $\alpha = 1$, $\alpha <1$, respectively, are satisfied.
\end{theo}


\begin{proof}[Proof of Theorem \ref{rfieldsolution}.]There exists a set $\tilde \Omega \subset \Omega$ of probability one such that for all $\omega  \in \tilde \Omega$, the function $t\mapsto Y_t(\omega)$ is locally integrable. Without loss of generality, we can suppose that $\Omega=\tilde \Omega$. Let $\varphi \in \m D(\R^d)$ be such that $\varphi \s 0$, $\text{supp}\, \varphi \subset B(0,1)$ and $\int_{\R^d} \varphi =1$. For each $t \in \R^d$ and $n\in \N$, we define $\varphi_n^t (\cdot) = n^d \varphi (n(\cdot-t))$. Let $Z^n_t(\omega):= \scal{Y(\omega)}{\varphi_n^t}$. Then
\begin{equation}\label{int_form}
 Z^n_t(\omega)=\int_{\R^d} Y_s (\omega) n^d \varphi (n(s-t)) \dd s= \int_{\R^d} Y_{r+t} (\omega) n^d \varphi (n r) \dd r \, .
\end{equation}
Define $f(t,s,\omega) := (t+s, \omega)$. The function $f$ is measurable as a map from $( \R^d\times \R^d\times \Omega , B(\R^d)\otimes B(\R^d)\otimes \m F) $ to $(\R^d\times \Omega, B(\R^d)\otimes \m F)$, and $Y_{r+t}(\omega) = Y\circ f (r,t,\omega)$. Since $Y$ is a jointly measurable process, and by Fubini's theorem, we deduce from the second equality in \eqref{int_form} that $Z^n$ is a jointly measurable process. We define the set
\begin{equation*}
 A=\left \{ (t,\omega) : \scal{Y(\omega)}{\varphi_n^t} \to Y_t(\omega) \ \text{as} \ n \to +\infty \right\}\, .
\end{equation*}
We can write
\begin{equation*}
 A=\bigcap_{k\in \N*}\bigcup_{N\in \N}\bigcap_{n\s N} \left\{ (t,\omega) : \left | Z^n_t(\omega)- Y_t(\omega) \right | \I \frac 1 k \right \}\, ,
\end{equation*}
and since $Z^n$ and $Y$ are both jointly measurable processes, $A\in \m B(\R^d)\otimes \m F$. By Lebesgue's differentiation theorem (see \cite[Chapter 7, Exercise 2]{zygmund}), for any $\omega \in \Omega$, $\int_{\R^d} \mathds 1_{(t,\omega) \in A^c} \dd t =0$.
 Then, by Fubini's theorem,
there is a non random set $\tilde A \subset \R^d$ such that $\leb d {\tilde A} =0$ and for all $t\notin \tilde A$, $\PR\left\{\omega : (t,\omega) \in A^c\right \}=0$, that is,
\begin{equation}\label{rd4}
\PR \left\{ \scal{Y}{\varphi_n^t} \to Y_t \ \text{as} \ n\to +\infty \right \}=1.
\end{equation}

By \cite[Proposition 3.4.1]{taqqu}, for any $f\in L^\alpha (S)$,
\begin{equation}\label{charfunction}
 \E\lp e^{i \scal{\dot W^\alpha}{f}} \rp =e^{ -\|f\|_{L^\alpha(S)}^\alpha }\, ,
\end{equation}
 where $\|f\|_{L^\alpha(S)}^\alpha=\int_{S} |f(x)|^\alpha \dd x$. Therefore, by \eqref{def1} and \eqref{eqdef},
\begin{equation}\label{fouriergen}
 \E\lp e^{i \scal{u_{\text{gen}}}{\varphi}} \rp =e^{-\|\varphi * \check\rho \|_{L^\alpha}^\alpha}= \E\lp \exp\lp i \int_{\R^d} Y_s \varphi (s) \dd s\rp \rp\, .
\end{equation}
Let $t_0\in \tilde A^c$. Then $\scal{Y}{\varphi_n^{t_0}} \to Y_{t_0}$ almost surely as $n\to +\infty$. We define $\rho_n^{t_0}= \varphi_n^{t_0} * \check\rho \in L^\alpha(S)$ by  \hyperlink{hyp1}{\textbf{(H1)}}. By \eqref{fouriergen}, for $n,m\in \N$,
\begin{equation}\label{chifunctioncv}
 e^{-\|\rho_n^{t_0}- \rho_m^{t_0} \|_{L^\alpha}^\alpha}= \E\lp \exp\lp i \int_{\R^d} Y_s \lp \varphi_n^{t_0}(s) -\varphi_m^{t_0}(s)\rp \dd s \rp \rp \to 1 \qquad \text{as} \ n,m \to +\infty \, .
\end{equation}
We deduce that $(\rho_n^{t_0})_{n\s 1}$ is a Cauchy sequence in $L^\alpha(S)$. By completeness of this space, there is a function $ g^{t_0} \in L^\alpha(S)$ such that
\begin{equation}\label{truc1}
 \rho_n^{t_0} \to  g^{t_0} \, , \qquad \text{in} \ L^\alpha(S) \ \text{as} \ n\to +\infty\, .
\end{equation}
Furthermore, we know from the theory of generalized functions that $\varphi_n^{t_0}\to \delta_{t_0}$ in $\m D'(\R^d)$ as $n\to +\infty$. Therefore,
\begin{equation}\label{truc2}
 \rho_n^{t_0} \to \delta_{t_0}* \check\rho \, , \qquad \text{in} \ \m D'(\R^d) \ \text{as} \ n\to +\infty\, .
\end{equation}
From \eqref{truc1} and \eqref{truc2}, we would like to deduce that $\delta_{t_0}* \check \rho=g^{t_0}$ in $\m D'(S)$. 
If this equality is true, then it means that $s\mapsto \rho(t_0-s)$ can be considered as a function in $L^\alpha(S)$. However, in order to prove this equality, it suffices to show that for any $\theta \in \m D(S)$, $\scal{\delta_{t_0}*\check\rho}{\theta}=\scal{ g^{t_0}}{\theta}$.

In the case $\alpha \s 1$, by Hölder's inequality,
$$|\scal{g^{t_0} - \rho_n^{t_0}}{\theta}|\I \int_S \left |  g^{t_0} (s)- \rho_n^{t_0}(s) \right | |\theta(s)| \dd s\I \|  g^{t_0}- \rho_n^{t_0} \|_{L^\alpha(S)} \| \theta \|_{L^{\frac{\alpha}{\alpha-1}}(S)}\, .$$
Passing to the limit as $n\to+\infty$, we get that for all $t_0\in \tilde A^c$, $\delta_{t_0}*\check\rho=  g^{t_0} \in L^\alpha(S)$ in $\m D'(S)$. Then, in the sense of distributions, $\check \rho=\delta_{-t_0}*\delta_{t_0}*\check \rho=\delta_{-t_0}*g^{t_0}$. Therefore, in the sense of distributions, $\rho$ is equal to the function $t\in\R^d \mapsto g^{t_0}(t_0-t)$, which therefore does not depend on $t_0$, and is such that for almost all $t\in \R^d, \, \delta_t*\check \rho=\rho(t-\cdot)\in L^{\alpha}(S)$. Also, for any $t\in \tilde A^c$, by \eqref{eqdef},
$$\scal{Y}{\varphi_n^{t}} =\scal{u_{\text{gen}}}{\varphi_n^{t}} =\scal{\dot W^\alpha}{\varphi_n^t*\check \rho}=\scal{\dot W^\alpha}{\rho_n^t}\, ,$$
 and $\scal{Y}{\varphi_n^{t}} \to Y_{t}$ almost surely as $n\to +\infty$, and
\begin{equation}\label{gtfunction}
  \rho_n^t\to g^t=\delta_t*\check \rho\, , \qquad \text{in} \ L^\alpha(S)\  \text{as} \ n\to +\infty\, .
\end{equation}
Therefore, $\scal{\dot W^\alpha}{\rho_n^t - \rho(t-\cdot)} \to 0$ in law, hence in probability, that is,
 $\scal{\dot W^\alpha}{\rho_n^t} \to \scal{\dot W^\alpha}{\rho(t-\cdot)}$ in probability as $n\to +\infty$. Therefore \eqref{mildresult} holds. Since we used Hölder's inequality, this method does not work in the case $\alpha <1$, and does not imply \eqref{alpha0}, \eqref{alpha1} or \eqref{alphag1}. We therefore develop a different argument that works for any $\alpha \in (0,2)$.

If $\alpha\in (0,2)$ is arbitrary, we deduce from \eqref{charfunction} and \eqref{chifunctioncv} that $\scal{\dot W^\alpha}{\rho_n^{t_0}-g^{t_0}}\to 0$ in law as $n\to +\infty$, and by \cite[Lemma 4.7]{kallenberg}, the convergence is also in probability. By almost sure linearity, we deduce that $\scal{\dot W^\alpha}{\rho_n^{t_0}}\to \scal{\dot W^\alpha}{g^{t_0}}$ in probability as $n\to +\infty$. By uniqueness of the limit, and since $\scal{\dot W^\alpha}{\rho_n^{t_0}}=\scal{u_{\text{gen}}}{\varphi_n^{t_0}}=\scal{Y}{\varphi_n^{t_0}}$, it follows that
\begin{equation}\label{rd5}
 Y_{t_0} = \scal{\dot W^\alpha}{g^{t_0}} \, , \qquad \text{a.s. for any} \ t_0 \in \tilde A^c\, .
\end{equation}
For any $(t,s) \in \R^d\times S$, let
\begin{equation}\label{gfunction}
 g(t,s)=\limsup_{n\to +\infty}\, \rho_n^t(s)\, .
\end{equation}
Then $(t,s)\mapsto g(t,s)$ is measurable, and for $t\in \tilde A^c, \ g(t,\cdot)=g^t(\cdot)$ almost everywhere. Therefore
\begin{equation*}
 Y_{t_0} = \scal{\dot W^\alpha}{g(t_0, \cdot)} \, , \qquad \text{a.s. for any} \ t_0 \in \tilde A^c\, .
\end{equation*}
Let $\psi \in \m D(\R^d)$. Then, $\mu_\psi (\!\dd t):=  \psi(t) \dd t$ is a finite signed measure, that we can decompose into positive and negative parts $\mu^+_\psi$ and $\mu_\psi^-$. Since $Y$ is almost surely locally integrable,
\begin{equation*}
 \int_{\R^d} |Y_t| \mu^+_\psi(\!\dd t)<+\infty\, , \qquad \text{and} \qquad  \int_{\R^d} |Y_t| \mu^-_\psi(\!\dd t)<+\infty \qquad \text{a.s.}
\end{equation*}
By \cite[Theorem 11.3.2]{taqqu}, if $\alpha>1$, we get
\begin{equation}\label{e1}
  \int_{\R^d}\lp \int_{S} \left | g(t,s) \right |^\alpha \dd s \rp^{\frac 1 \alpha}  |\psi(t)|\dd t <+\infty\, ,
\end{equation}
if $\alpha=1$, we get
\begin{equation}\label{e2}
  \int_{\R^d}\dd t  \int_{S} \! \dd s \, | g(t,s) \psi(t)| \left [ 1+ \log_+ \lp \frac{|g(t,s)| \int_{\R^d}\int_S |g(r,v)| \dd v |\psi(r)| \dd r   }{\lp\int_S |g(t,v)| \dd v \rp\lp \int_{\R^d} |g(r,s) \psi(r) | \dd r\rp} \rp \right ] <+\infty\, ,
\end{equation}
and if $\alpha<1$, we get
\begin{equation}\label{eq3}
 \int_{S}\lp \int_{\R^d} \left |g(t,s) \psi(t) \right | \dd t \rp^{\alpha} \dd s <+\infty \, .
\end{equation}
By the generalized Minkowsky inequality (see \cite[A.1]{stein}) and by \eqref{e1}, when $\alpha>1$,
$$\lp \int_S \left | \int_{\R^d} \left | g(t,s) \psi(t) \right | \dd t \right |^\alpha \dd s\rp^{\frac 1 \alpha} \I \int_{\R^d}\lp \int_{S} \left | g(t,s) \right |^\alpha \dd s \rp^{\frac 1 \alpha}  |\psi(t)|\dd t <+\infty\, .$$
In particular, we see that  for almost all $s\in S$, $t\mapsto g(t,s)$ is locally integrable (and therefore defines a distribution). By \eqref{rd5},
\begin{equation}\label{eq1}
 \int_{\R^d} Y_t \mu_{\psi}(\! \dd t) = \int_{\R^d} \scal{\dot W^\alpha}{g(t,\cdot)} \psi(t) \dd t
\end{equation}
By \cite[Theorem 11.4.1]{taqqu}, we can exchange the stochastic integral with the Lebesgue integral in \eqref{eqdef}, to see that this equals
\begin{equation}\label{rd6}
  \scal{\dot W^\alpha}{ \int_{\R^d} \psi(t)g(t,\cdot) \dd t} \qquad \text{a.s.}
\end{equation}

We define $\int_{\R^d} \psi(t)g(t,s) \dd t=: (\psi\circledast g) (s)$ (this operation on $\psi$ and $g$  is not commutative). From \eqref{rd6} and \eqref{eqdef}, we get
\begin{equation}\label{tructructruc}
  \scal{\dot W^\alpha}{\psi\circledast g}- \scal{\dot W^\alpha}{ \psi*\check \rho}= \scal{\dot W^\alpha}{ \psi\circledast g-\psi*\check \rho}=\scal{Y}{\psi}-\scal{u_{\text{gen}}}{\psi}=0 \ \text{a.s.},
\end{equation}
and by \eqref{charfunction}, we deduce that $\|\psi\circledast g-\psi*\check \rho\|_{L^\alpha}^\alpha=0$. Then, for any $\psi \in \m D(\R^d)$, there is a set $B_\psi$ such that $\leb d {B_\psi}=0$ and for any $s\in S\setminus B_\psi$, $(\psi\circledast g)(s)= (\psi*\check \rho)(s)$.
Since $\m D(\R^d)$ is separable, there is a countable dense subset $D\subset \m D(\R^d)$. Let
\begin{equation*}
 B=\bigcup_{\psi \in D}B_{\psi} \, , \qquad \leb d B =0\, .
\end{equation*}
Then, for all $s\in S\setminus B$, for all $\psi \in D$,
$$\scal{g(\cdot, s)}{\psi}=\psi \circledast g(s)= \psi *\check \rho(s)=\scal{ \rho}{\psi(s+\cdot)}=\scal{\delta_s*\rho}{\psi}\, ,$$
where we have used \eqref{rd2}. Since two distributions equal on a dense set are equal everywhere by continuity, we get that for all $s\in S\setminus B$, $g(\cdot, s)=\delta_s*\rho$ in $\m D'(\R^d)$. Then, $\rho=\delta_{-s}*g(\cdot, s)$ in $\m D'(S)$, and $\rho$ is a function depending only on the $t\in \R^d$ variable, more precisely for almost all $t\in \R^d$, $\rho(t)=g(t+s,s)$ which does not depend on $s$. Then, for almost all $(t,s)\in \R^d\times S$, $g(t,s)=\rho(t-s)$. By definition of $g$ in \eqref{gfunction} and by \eqref{gtfunction}, we deduce that $\rho$ is a function such that for almost all $t\in \R^d$, $\rho(t-\cdot) \in L^\alpha(S)$. Also, from \eqref{e1}, \eqref{e2} and \eqref{eq3}, we get 
that \eqref{alpha0}--\eqref{alphag1} hold.
\end{proof}

\begin{rem}
 The proof of Theorem \ref{rfieldsolution} in the case $\alpha \s 1$ shows that the result is still valid in the case of Gaussian white noise: it suffices to set $\alpha=2$ in the proof.
\end{rem}


\subsection{The case of symmetric pure jump Lévy noise.}

We now consider the more general case of a symmetric pure jump Lévy noise $\dot X$. Similarly to the $\alpha$-stable case, we can obtain a necessary condition for the existence of a random field solution.

\begin{theo}\label{rfieldsolution2}
Assume  \hyperlink{hyp1}{\textbf{(H1)}}. Let $u_{\text{gen}}$ be the generalized solution to \eqref{equation} defined by \eqref{def1}. Suppose that $u_{\text{gen}}$ has a random field representation $Y$ in the sense of Definition \ref{defrfrepresentation}, that is, there exists a jointly measurable random field $(Y_t)_{t\in \R^d}$ such that $Y$ has almost surely locally integrable sample paths, and for any $\varphi\in \m D(\R^d)$,
\begin{equation*}
 \scal{u_{\text{gen}}}{\varphi}= \int_{\R^d} Y_t\, \varphi(t) \dd t \qquad \text{a.s.}
\end{equation*}
Then, for almost all $t\in \R^d$, $\rho(t-\cdot) \in L( \dot X,S )$ (i.e.  \hyperlink{hyp2}{\textbf{(H2)}} is satisfied almost everywhere), and
\begin{equation*}
 Y_t =\scal{\dot X}{\rho(t-\cdot)}=u_{\text{mild}}(t) \qquad \text{a.s.} \quad \text{a.e.}
\end{equation*}
\end{theo}

\begin{proof}
We use the same notations as in the proof of Theorem \ref{rfieldsolution}. By the same reasoning as in the proof of Theorem \ref{rfieldsolution}, there is a non random set $\tilde A \subset \R^d$ such that $\leb d {\tilde A} =0$ and for all $t\notin \tilde A$, $\PR \left\{ \scal{Y}{\varphi_n^t} \to Y_t \ \text{as} \ n\to +\infty \right \}=1$. Then, as before we define $\rho_n^{t_0}= \varphi_n^{t_0} * \check\rho \in L(\dot X,S)$.
For $n,m\in \N$ and $t_0 \notin \tilde A$,
\begin{align*}
\E \lp e^{i\scal {\dot X} {\rho_n^{t_0}-\rho_m^{t_0}}} \rp= \E\lp e^{i \int_{\R^d} Y_s \lp \varphi_n^{t_0}(s) -\varphi_m^{t_0}(s)\rp \dd s} \rp \to 1 \qquad \text{as} \ n,m \to +\infty \, .
\end{align*}
We deduce that $\scal {\dot X} {\rho_n^{t_0}-\rho_m^{t_0}}$ converges to zero in law, hence in probability. Since $\dot X$ is symmetric, the linear mapping $f\in L(\dot X, S) \mapsto \scal {\dot X} f \in L^0(\Omega)$ is an isomorphism (see \cite[Theorem 3.4 and Proposition 3.6]{rosinski}). In particular the inverse map is continuous, therefore the sequence $\lp \rho_n^{t_0}\rp_{n\in \N}$ is Cauchy in $L(\dot X, S)$. This space is complete, therefore there is a function $g^{t_0}$ such that $\rho_n^{t_0}\to g^{t_0}$ in $L(\dot X, S)$. For any $(t,s) \in \R^d\times S$, let
\begin{equation*}
 g(t,s)=\limsup_{n\to +\infty}\, \rho_n^t(s)\, .
\end{equation*}
Then $(t,s)\mapsto g(t,s)$ is measurable, and for $t\in \tilde A^c, \ g(t,\cdot)=g^t(\cdot)$ almost everywhere. As in \eqref{rd5}, we get that for almost all $t_0 \in \R^d$, $Y_{t_0}=\scal{\dot X}{g^{t_0}}$ almost surely. Also, since $Y$ has almost surely locally integrable sample paths, for any $\psi\in \m D(\R^d)$,
\begin{equation*}
 \int_{\R^d} \left | Y_t \right | \mu_\psi(\! \dd t)  =    \int_{\R^d} \left | \scal{\dot X}{g^{t}} \right | \mu_\psi(\! \dd t)<+\infty \qquad \text{a.s.}\, ,
\end{equation*}
where $\mu_\psi(\! \dd t)=|\psi(t) | \dd t$. By Theorem \ref{generalfubini}, as in \eqref{rd6},
\begin{equation*}
  \int_{\R^d}  \scal{\dot X}{g^{t}} \psi(t)  \dd t= \scal{\dot X}{\psi \circledast g} \qquad \text{a.s.}
\end{equation*}
Therefore, for any $\psi\in \m D(\R^d)$,
\begin{equation*}
 \scal{\dot X}{\psi \circledast g}= \int_{\R^d}  \scal{\dot X}{g^{t}} \psi(t)  \dd t= \int_{\R^d}  Y_t \psi(t)  \dd t=  \scal{\dot X}{\check \rho*\psi} \qquad \text{a.s.,}
\end{equation*}
where the last equality is by Definitions \ref{defrfrepresentation} and \ref{gensol}. Therefore, for almost every $s\in S$, $\psi \circledast g(s)=\psi*\check \rho(s)$. We can then conclude as in the proof of Theorem \ref{rfieldsolution} after \eqref{tructructruc}.
\end{proof}


\section{Examples}\label{examples}

In this section, we give some examples of application of Theorems \ref{mildisweak} and \ref{rfieldsolution}. We focus on three well-known stochastic partial differential equations: the linear stochastic heat equation, the linear stochastic wave equation, and the linear Poisson equation, in all spatial dimensions. We restrict to the case of a symmetric $\alpha$-stable noise, as the choice of the parameter $\alpha \in (0,2)$ will be enough to capture the different cases.


\subsection{The stochastic heat equation}
Let $\nalpha W$ be an $\alpha$-stable symmetric noise on $\R_+\times \R^d$. The heat operator $\sH$ in dimension $d$ is a constant coefficient partial differential operator given by
$$\sH=\frac{\partial}{\partial t}-\sum_{i=1}^d \frac{\partial^2}{\partial x_i^2}\, .$$
A fundamental solution $\rho_\sH$ for this operator is given by the formula
\begin{equation*}
 \rho_\sH(t,x)=\frac{1}{\lp 4\pi t\rp^{\frac d 2}}\exp \lp -\frac{|x|^2}{4t} \rp \mathds 1_{t>0}\, .
\end{equation*}
We consider the following Cauchy problem
\begin{equation}\label{SHEalpha}
\left\{\begin{array}{l}  \sH u=\nalpha W , \\ u(0,\cdot)=0  . \end{array}\right.
\end{equation}


\subsubsection{Existence of a generalized solution}
We wish to define the generalized solution of this equation associated with the fundamental solution $\rho_\sH$.

\begin{prop}\label{rd9}
 For any choice of $\alpha \in (0,2)$ and $d\s 1$, the generalized solution to the linear stochastic heat equation driven by a symmetric $\alpha$-stable noise is well defined.
\end{prop}

\begin{proof}
 We have to check for which combinations of $\alpha$ and $d$ the convolution $\varphi* \check \rho_\sH$ belongs to $L^\alpha(\R_+\times \R^d)$, for any $\varphi \in \m D(\R^{d+1})$ (see  \hyperlink{hyp1prime}{\textbf{(H1')}}). For $\varphi \in \m D(\R^{d+1})$ and $(t,x)\in \R\times \R^d$,
$$\varphi* \check \rho_\sH(t,x)= \int_t^{+\infty} \! \dd s \int_{\R^d} \! \dd y \frac{1}{\lp 4\pi (s-t)\rp^{\frac d 2}}\exp \lp -\frac{|y-x|^2}{4(s-t)} \rp \varphi (s,y)\, .$$
Since $\varphi$ has compact support, we see from this formula that there is a $T\in \R_+$ such that for any $t\s T$ and $x\in \R^d$, $\varphi* \check \rho_\sH(t,x)=0$. Therefore, we need to check that $\varphi* \check \rho_\sH$ is in $L^\alpha([0,T] \times \R^d)$ for any $T\in \R_+$ and $\varphi \in \m D(\R^{d+1})$. The function $\varphi* \check \rho_\sH$ is smooth, so we only need to check integrability for $x$ in a neighborhood of infinity. Then, for some compact $K \subset \R^d$, for $x$ large enough,
\begin{align*}
 |\varphi* \check \rho_\sH(t,x)| \I |\varphi|* \check \rho_\sH(t,x) &= \mathds 1_{t\I T} \int_t^{T} \! \dd s \int_{K} \! \dd y   \frac{1}{\lp 4\pi (s-t)\rp^{\frac d 2}}\exp \lp -\frac{|y-x|^2}{4(s-t)} \rp \left | \varphi (s,y)\right | \\
 &\I \mathds 1_{t\I T} \|\varphi \|_{\infty} \int_t^{T} \! \dd s \int_{K} \! \dd y   \frac{1}{\lp 4\pi (T-t)\rp^{\frac d 2}}\exp \lp -\frac{|y-x|^2}{4(T-t)} \rp \, ,
\end{align*}
where the second inequality comes from the fact that for $|x|$ large enough, the function $s \in [t, T] \mapsto  \frac{1}{\lp 4\pi (s-t)\rp^{\frac d 2}}\exp \lp -\frac{|y-x|^2}{4(s-t)}\rp$ is non-decreasing and realizes its maximum at $s=T$. Then, using the inequality
\begin{equation}\label{ineq1}
 |y-x|^2 \s \frac 1 2 |x|^2 -|y|^2 \, ,
\end{equation}
we get
\begin{align*}
  | \varphi * \check \rho_\sH(t,x)  | &\I \mathds 1_{t\I T}  \frac{\|\varphi \|_{\infty}}{\lp 4 \pi \rp^{\frac d 2 }} (T-t)^{-\frac d 2 +1} \int_{K} \! \dd y  \exp \lp -\frac{|y-x|^2}{4(T-t)} \rp \\
 &\I \mathds 1_{t\I T}  \frac{\|\varphi \|_{\infty}}{\lp 4 \pi \rp^{\frac d 2 }} (T-t)^{-\frac d 2 +1} \exp\lp -\frac{|x|^2}{8(T-t)} \rp \int_K \exp \lp-\frac{|y|^2}{4(T-t)} \rp  \dd y \, .
\end{align*}
We evaluate the integral and deduce that
\begin{align}\label{majoration}
   | \varphi * \check \rho_\sH(t,x) |  &\I  \mathds 1_{t\I T} \|\varphi \|_{\infty}T \exp\lp -\frac{|x|^2}{8T}\rp \, .
\end{align}
From \eqref{majoration} we deduce that $\varphi* \check \rho_\sH$ has compact support in the time variable (uniformly with respect to the space variable), and has rapid decay in the space variable. Therefore $\varphi* \check \rho_\sH\in L^\alpha([0,T]\times \R^d)$ for any $\alpha \in \R_+$. We deduce that the stochastic linear heat equation driven by symmetric $\alpha$-stable noise always has a generalized solution $u_{\text{gen}}$ defined by
\begin{equation}\label{solutionSHE}
 \scal{u_{\text{gen}}} \varphi := \scal{\dot W^\alpha}{\varphi* \check \rho_\sH} \, , \qquad \text{for all} \ \varphi \in \m D(\R^{d+1}).
\end{equation}
\end{proof}
\begin{rem}
 The previous proof is still valid if we formally replace $\alpha$ by $2$, and therefore the same result is true in the Gaussian case.
\end{rem}

\begin{rem}
 From \eqref{majoration}, we get that
\begin{equation*}
 \| \varphi * \check\rho_\sH\|_{L^\alpha([0,T]\times \R^d)} \I C \| \varphi \|_\infty\, ,
\end{equation*}
for some constant $C$ that depends on the support of $\varphi$. Therefore, if $\varphi_n$ is a sequence of test functions in $\m D(\R^{d+1})$ such that $\varphi_n \to 0$ in $\m D(\R^{d+1})$, then
\begin{equation*}
 \E \lc e^{i\xi \scal{u_{\text{gen}}}{\varphi_n}}\rc=e^{-|\xi|^\alpha  \| \varphi_n * \check\rho_\sH\|_{L^\alpha([0,T]\times \R^d)}^\alpha}\to 1 \, , \qquad \text{as} \ n \to +\infty\, .
\end{equation*}
Therefore, $\scal{u_{\text{gen}}}{\varphi_n}\to 0$ in law as $n \to +\infty$, and since convergence in law to a constant is equivalent to the convergence in probability to this constant, we deduce that $\scal{u_{\text{gen}}}{\varphi_n}\to 0$ in probability as $n\to +\infty$. Therefore, $u_{\text{gen}}$ defines a linear functional on $\m D(\R^{d+1})$ that is continuous in probability. The space $\m D(\R^{d+1})$ is nuclear (see \cite[p. 510]{treves}), so by \cite[Corollary 4.2]{spdewalsh}, $u_{\text{gen}}$ has an almost surely continuous version (and therefore $u_{\text{gen}}$ defines a continuous generalized stochastic process).
\end{rem}


\subsubsection{Existence of a mild solution}
The criterion for the existence of the mild solution to the  linear stochastic heat equation \eqref{SHEalpha} is known (see \cite{balan}). However, we can also obtain this from \hyperlink{hyp2prime}{\textbf{(H2')}}.
\begin{prop}\label{rd8}
  The mild solution to the linear stochastic heat equation driven by a symmetric $\alpha$-stable noise, as defined in \eqref{def}, exists if and only if
  \begin{equation}\label{alphacondition}
 \alpha < 1+\frac 2 d\, .
\end{equation}
In this case,
\begin{equation}\label{solutionmildheat}
 u_{\text{mild}}(t,x):= \scal{\dot W^\alpha}{\rho_\sH(t-\cdot, x-\cdot)}\, .
\end{equation}
\end{prop}

\begin{proof}
The mild solution of \eqref{SHEalpha} associated with $\rho_\sH$ is well defined if and only if the following integral is finite for any $(t,x)\in \R_+ \times \R^d$ (see  \hyperlink{hyp2prime}{\textbf{(H2')}}):
\begin{align}\nonumber
 \int_{\R_+} \! \dd s \int_{\R^d} \! \dd y \, \rho_\sH (t-s,x-y)^\alpha &= \int_0^t \! \dd s \frac{1}{\lp 4\pi s\rp^{\alpha \frac d 2}} \int_{\R^d} \! \dd y \exp \lp -\frac{\alpha |y|^2}{4s} \rp \\
 &= \int_0^t \! \dd s \frac{1}{\lp 4\pi s\rp^{ \frac d 2\lp \alpha - 1 \rp} \alpha^{\frac d 2}} \, ,
 \label{rd7}
\end{align}
and the last integral is finite if and only if
\begin{equation*}
 \alpha < 1+\frac 2 d\, .
\end{equation*}
In this case, by Definition \ref{milddef},
\begin{equation*}
 u_{\text{mild}}(t,x):= \scal{\dot W^\alpha}{\rho_\sH(t-\cdot, x-\cdot)}\, .
\end{equation*}
\end{proof}


\subsubsection{Existence of a random field solution}

We have seen in the previous section that for any $\alpha$ and $d$, it is possible to define the generalized solution $u_{\text{gen}}$, and that the mild solution $u_{\text{mild}}$ exists if and only if $\alpha < 1+ \frac 2 d$. We now apply the results of Theorem \ref{mildisweak} and Theorem \ref{rfieldsolution} to learn more about the relations between those two notions of solution.
\begin{prop}\label{rd10}
The generalized solution $u_{\text{gen}}$ to the linear stochastic heat equation driven by a symmetric $\alpha$-stable noise has a random field representation $Y$ if and only if \eqref{alphacondition} is satisfied, and in that case, this random field representation $Y$ is equal to $u_{\text{mild}}$ almost everywhere almost surely.
\end{prop}

\begin{proof}
If $\alpha \in \lp 1 , 1+\frac d 2 \rp$, then from Theorem \ref{mildisweak}\textit{(i)}, we deduce that $u_{\text{mild}}$ is almost surely equal to $u_{\text{gen}}$ (the condition \eqref{alpha0} is immediately verified using \eqref{rd7} and Remark \ref{mildisweakrem}(1)).
Similarly, for any $\varphi \in \m D(\R^{d+1})$, if $\alpha<1$, then by \eqref{majoration}, $|\check \rho_\sH|*|\varphi| \in L^\alpha(\R_+\times \R^d)$, hence by Theorem \ref{mildisweak}\textit{(iii)}, the mild solution of the stochastic heat equation $u_{\text{mild}}$ is equal to the generalized solution $u_{\text{gen}}$.

The case $\alpha=1$ is slightly more involved, since we need to check condition \eqref{alpha1}. Let $\varphi \in \m D(\R^{d+1})$. First, we have
$$\int_{\R_+\times \R^d}  \rho_\sH(t-s, x-v) \dd v=t \mathds 1_{t> 0} \, ,$$
and for any $x\in \R_+$, $\log_+(x) \I | \log (x) |$, therefore, for $t>0$
\begin{align*}
  \log_+ & \lp \frac{ \rho_\sH(t-s,x-y) \int_{\R^{d+1}}\int_{\R_+\times \R^d} | \rho_\sH(u-v, r-w)| \dd v \dd w \, \mu_\varphi(\!\dd u, \! \dd r)   }{\lp\int_{\R_+\times \R^d}  \rho_\sH(t-v,x-w) \dd v \dd w \rp\lp \int_{\R^{d+1}}  \rho_\sH( u-s, r-y)  \mu_\varphi( \!\dd u, \! \dd r)\rp} \rp \\
  &\hspace{1cm} \I \left | \log \lp  \rho_\sH(t-s,x-y) \rp \right | +  \left | \log \lp \int_{\R^{d+1}} u \mu_\varphi(\!\dd u, \! \dd r) \rp \right |+ \left | \log (t) \right | \\
  & \hspace{2cm} + \left | \log \lp  \check \rho_\sH * | \varphi | (s,y) \rp \right | \, .
\end{align*}
Hence, to have \eqref{alpha1}, we need to check the finiteness of the following integrals:
\begin{align*}
 I_I&:= \int_{\R_+\times \R^d}\lp \check \rho_\sH* |\varphi| \rp (s,y) \dd s \dd y \, ,\\
 I_2&:=\int_{\R_+\times \R^d} \left(\lp \check \rho_\sH\left | \log \lp \check \rho_\sH\rp \right | \rp* |\varphi|\right) (s,y) \dd s \dd y\, ,\\
 I_3&:=\int_{\R_+\times \R^d} \lp \int_{\R^{d+1}}   \rho_\sH (t-s,x-y) |\log(t) \varphi(t,x) | \dd t \dd x \rp \dd s \dd y \, ,\\
 I_4&:=\int_{\R_+\times \R^d}  \left | \log \lp \check \rho_\sH * | \varphi | (s,y) \rp \right |\lp \check \rho_\sH * | \varphi |\rp (s,y) \dd s \dd y\, .
\end{align*}
The case of $I_1$ has already been treated after \eqref{majoration}, and for $I_3$, we can simply permute the integrals and get
$$I_3=\int_{\R^{d+1}} | t \mathds 1_{t> 0} \log(t) \varphi(t,x) |  \dd t \dd x<+\infty \, .$$
For $I_2$ and $I_4$, by the same considerations as for the case $\alpha\neq 1$, we need to check that for any $\varphi \in \m D(\R^{d+1})$,
\begin{align*}
 &(t,x)\in \R_+\times \R^d \mapsto \left | \check\rho_\sH \log (\check\rho_\sH)  \right | * |\varphi| (t,x) \, ,
 \end{align*}
 and
 \begin{align*}
 \ &(t,x)\in \R_+\times \R^d \mapsto \lp \check\rho_\sH * |\varphi| \rp (t,x) \left | \log (\check\rho_\sH * | \varphi | (t,x)) \right |\, ,
\end{align*}
are in $L^1([0,T]\times \R^d)$ for any $T\in \R_+$. By \eqref{majoration}, we get that $\lp \check\rho_\sH * |\varphi| \rp \left | \log (\check\rho_\sH * | \varphi | ) \right | \in L^1(\R_+ \times \R^d)$, therefore $I_4 < +\infty$. We now turn to $I_2$. Observe that
\begin{align*}
  \left | \check\rho_\sH \log (\check\rho_\sH)  \right |*| \varphi|(t,x)&= \mathds 1_{t\I T}\int_t^{T} \! \dd s \int_{K} \! \dd y \frac{1}{\lp 4\pi (s-t)\rp^{\frac d 2}}\exp \lp -\frac{|y-x|^2}{4(s-t)} \rp \\
 & \hspace{3cm} \times \left | -\frac d 2 \log \lp 4 \pi (s-t) \rp -\frac{|y-x|^2}{4(s-t)} \right | |\varphi (s,y)|\, .
\end{align*}
Again, by continuity (since $|\varphi|$ is continuous and has compact support), we are only concerned about integrability near a neighborhood of infinity ($x \to \pm \infty$, $t \in [0,T]$). Since $\log \lp 4 \pi (s-t) \rp$ is integrable at $s=t$ and the polynomial term $|y-x|^2/(s-t)$ barely affects the decay as $x \to \pm \infty$ of $\exp \lp -|y-x|^2/(4(s-t)) \rp$, we can obtain a bound similar to \eqref{majoration}: see \cite[Proof of Proposition 4.4.4]{humeau1} for details.

Hence by Theorem \ref{mildisweak}\textit{(ii)}, the mild solution $u_{\text{mild}}$ of the stochastic heat equation in the case $\alpha=1$ is also equal to $u_{\text{gen}}$.

Furthermore, if $u_{\text{gen}}$ has a random field representation $Y$ in the sense of Definition \ref{rfrepresentation}, then, by Theorem \ref{rfieldsolution}, necessarily $\rho_\sH \in L^\alpha ([0,T] \times \R^d)$ for any $T>0$, which, as we have seen in the proof of Proposition \ref{rd8}, is equivalent to \eqref{alphacondition}, and also the random field representation $Y$ is equal to the mild solution $u_{\text{mild}}$ almost everywhere a.s. Therefore, a necessary and sufficient condition for the existence of a random field solution to the stochastic heat equation \eqref{SHEalpha} is that $\alpha<1+\frac 2 d$.
\end{proof}

Propositions \ref{rd9}, \ref{rd8} and \ref{rd10} together establish the following theorem.

\begin{theo}\label{heat_result}
 The generalized solution $u_{\text{gen}}$ to the stochastic heat equation \eqref{SHEalpha} defined by \eqref{solutionSHE} always exists. The mild solution $u_{\text{mild}}$ defined by \eqref{solutionmildheat} exists if and only if
\begin{align} \label{alphacond}
 \alpha<1+\frac 2 d\, ,
\end{align}
Furthermore, a random field representation $Y$ of the generalized solution exists if and only if \eqref{alphacond} is satisfied and in this case, for almost all $(t,x) \in \R_+\times \R^d$,
\begin{equation*}
 Y_{t,x} =\scal{\dot W^\alpha}{\rho_\sH(t-\cdot, x-\cdot)}= u_{\text{mild}}(t,x) \qquad \text{a.s.} 
\end{equation*}
\end{theo}


\subsection{The stochastic wave equation}
We now consider the stochastic wave equation.
For an overview of this SPDE in the Gaussian case, see \cite{minicourse}. Let $\nalpha W$ be an $\alpha$-stable symmetric noise on $\R_+\times \R^d$. The wave operator $\sO$ in dimension $d$ is a constant coefficient partial differential operator given by
$$\sO=\frac{\partial^2}{\partial t^2}-\sum_{i=1}^d \frac{\partial^2}{\partial x_i^2}\, .$$
The fundamental solution of this operator (with support in the forward light cone) is a function only in dimension one and two. In dimension one, it is given by
\begin{equation*}
 \rho_1^\sO(t,x)=\frac 1 2 \mathds 1_{|x|\I t}\qquad \text{for all} \ (x,t)\in \R^2\, ,
\end{equation*}
and, in dimension two, by
\begin{equation*}
 \rho_2^\sO(t,x)=\frac 1 {2\pi} \frac 1 {\sqrt{t^2-|x|^2}} \mathds 1_{|x|<t}\qquad \text{for all} \ (t,x)\in \R\times \R^2\, .
\end{equation*}
In dimension $d\s 3$, the fundamental solution is a distribution that can be characterized by its Fourier transform in the space variable $x$ (see \cite{dalang99}).

This fundamental solution is related to the following Cauchy problem:
\begin{equation}\label{SWE}
\left\{\begin{array}{l}  \sO u=\nalpha W\, , \\ u(0,\cdot)=0\, , \\
\frac{\partial u}{\partial t}(0, \cdot)=0\, . \end{array}\right.
\end{equation}


\subsubsection{Existence of the generalized solution}
We first study the existence of the generalized solution in various dimensions $d\s 1$.

\begin{prop}\label{rdprop6.7}
 For any dimension $d\s 1$ and $\alpha \in (0,2)$, the generalized solution $u_{\text{gen}}$ to the linear stochastic wave equation driven by a symmetric $\alpha$-stable noise exists.
 \end{prop}

\begin{proof}
We need to check whether  \hyperlink{hyp1prime}{\textbf{(H1')}} is satisfied.
\vskip 12pt

\noindent
\underline{$d=1$:}
We need to check that for any $\varphi \in \m D(\R^2)$, the convolution $\varphi*\check \rho^\sO_1 $ is in $L^\alpha(\R_+\times \R)$. We get
$$\varphi*\check \rho^\sO_1 (t,x)= \int_0^{+\infty} \! \dd s \int_{-s}^{s} \! \dd y \,  \varphi (s+t,y+x) \, ,$$
and we can see from this expression that it is a smooth function with compact support, hence in $L^{\alpha}(\R_+\times \R)$.
\vskip 12pt

\noindent
\underline{$d=2$:}
Let $\varphi \in \m D(\R^3)$. We check whether or not for some $\alpha \in (0,2)$, the function $\varphi*\check \rho^\sO_2  \in L^\alpha(\R_+\times \R^d)$. By standard properties of the convolution, $\varphi*\check \rho^\sO_2$ is a smooth function. Let $(t,x)\in \R_+\times \R^2$. Then,
$$\varphi*\check \rho^\sO_2  (t,x)=\int_\R \! \dd s \int_{\R^2} \! \dd y \rho^\sO_2(s-t,y-x) \varphi(s,y)\, .$$
Since $\varphi$ has compact support and $\rho^\sO_2$ has support in the set $\{(t,x) \in \R_+\times \R^2 : |x|\I t\}$, we can write
$$\varphi*\check \rho^\sO_2 (t,x)=\mathds 1_{t\I T}\int_t^T \! \dd s \int_{B_x(T-t)} \! \dd y \rho^\sO_2(s-t,y-x) \varphi(s,y)\, ,$$
for some $T\in \R_+$, where $B_x(r)$ is the open ball of radius $r$ centered at $x$. We see in this expression that the convolution has compact support in space and time, since if $x$ is far enough from the support of $\varphi$, the integrand is zero. We deduce that for any $\alpha \in (0,2)$, $\varphi*\check \rho^\sO_2  \in L^\alpha(\R_+\times \R^d)$, and the generalized solution to the stochastic linear wave equation in dimension $2$ always exists.
\vskip 12pt

\noindent
\underline{$d\s3$:} For any $\varphi \in \m D(\R\times \R^d)$, the function $\varphi*\check \rho^\sO_d$ is smooth.
By the same type of considerations on the support of the convolution $\varphi*\check \rho^\sO_d$ as in dimension one and two, we see that this function has compact support, therefore $\varphi*\check \rho^\sO_d\in L^\alpha(\R_+\times \R^d)$ for any $\alpha \in (0,2)$.
\vskip 12pt

We conclude that the generalized solution always exists.
\end{proof}

\subsubsection{Existence of the mild solution}

\begin{prop}\label{rdprop6.8}
For all $\alpha\in (0,2)$, the mild solution to the stochastic wave equation driven by a symmetric $\alpha$-stable noise exists only in dimensions one and two.
\end{prop}

\begin{proof}

\noindent
\underline{$d=1$:}
There is a mild solution to the wave equation driven by $\alpha$-stable noise if and only if for any $(t,x)\in \R^+\times \R, \ \rho_1^\sO(t-\cdot, x-\cdot) \in L^\alpha(\R_+\times \R)$ (see  \hyperlink{hyp2prime}{\textbf{(H2')}}). Therefore, for any $T>0$, we need to check the finiteness of the integral
\begin{equation}\label{alphanorm}
 \int_0^T \dd t \int_\R \dd x \rho_1^\sO(t,x)^\alpha= \int_0^T \dd t \int_\R \dd x \,  \frac 1 {2^\alpha} \mathds 1_{|x|\I t} = \frac{T^2}{2^\alpha}\, .
\end{equation}
We deduce that the mild solution exists for any choice of $\alpha \in (0,2)$.
\vskip 12pt

\noindent
\underline{$d=2$:}
 The mild solution exists if and only if $ \rho_2^\sO(t-\cdot, x-\cdot) \in L^\alpha(\R_+\times \R)$ for any $(t,x)\in \R_+\times \R^2$ (see  \hyperlink{hyp2prime}{\textbf{(H2')}}). We have
\begin{align*}
 \| \rho_2^\sO (t-\cdot, x-\cdot) \|_{L^\alpha(\R_+\times \R^2)}^{\alpha \vee 1}&= \int_0^t \! \dd s \int_{\R^2 } \! \dd y \frac{1}{(2\pi)^\alpha \lp (t-s)^2-|x-y|^2 \rp^{\frac \alpha 2}}\\
 &=\frac 1 {(2\pi)^\alpha}\int_0^t \!\dd s \int_{|u|\I s} \! \dd u \frac{1}{(s^2- |u|^2)^{\frac \alpha 2}}\, .
\end{align*}
Changing to polar coordinates, we get
\begin{align*}
 \| \rho_2^\sO (t-\cdot, x-\cdot) \|_{L^\alpha(\R_+\times \R^2)}^{\alpha \vee 1}&=\frac 1 {(2\pi)^{\alpha-1}}\int_0^t \!\dd s \int_0^s \! \dd r \frac{r}{(s- r)^{\frac \alpha 2}(s+r)^{\frac \alpha 2}} \, .
\end{align*}
This integral is finite if and only if $\frac \alpha 2 <1$, that is $\alpha<2$.  We can further evaluate this integral and we get
\begin{align}\label{swealphanorm}
 \| \rho_2^\sO (t-\cdot, x-\cdot) \|_{L^\alpha(\R_+\times \R^2)}^{\alpha \vee 1}&=\frac 1 {(2\pi)^{\alpha-1}}\int_0^t \!\dd s \int_0^s \! \frac{\dd r}{2} \frac{2r}{(s^2- r^2)^{\frac \alpha 2}} \nonumber \\
 &= \frac 1 {(2\pi)^{\alpha-1}}\int_0^t \!\dd s \frac{s^{2-\alpha}}{2-\alpha}= \frac {t^{3-\alpha}}{(2\pi)^{\alpha-1}(2-\alpha)(3-\alpha)} \, .
\end{align}

Therefore, in dimension $2$, there is always a mild solution to the linear stochastic wave equation with $\alpha$-stable noise.
\vskip 12pt

\noindent
\underline{$d\s3$:}
Since fundamental solutions of the wave equation in dimension $d\s 3$ are not functions, there is no mild solution.
\end{proof}

\begin{rem}
From this proof, we can deduce the already known result in the Gaussian case (see \cite[p. 46]{minicourse}) that a mild solution to the linear stochastic wave equation only exists in spatial dimension one.
\end{rem}


\subsubsection{Existence of a random field solution}

\begin{prop}\label{rdprop6.10}
 The generalized solution $u_{\text{gen}}$ to the linear stochastic wave equation driven by a symmetric $\alpha$-stable noise has a random field representation if and only if $d\I 2$, and in that case, this random field representation is equal to $u_{\text{mild}}$ almost everywhere almost surely.
\end{prop}

\begin{proof}
\noindent
\underline{$d=1$:}
We will show that the mild solution exists and is equal to the generalized solution by using Theorem \ref{mildisweak} and the Remark \ref{mildisweakrem}. If $\alpha>1$, it suffices to check that $\{\|\rho_1^\sO(t-\cdot, x-\cdot)\|_{L^\alpha(\R_+\times \R)}, \, (t,x)\in \R_+\times \R \} \in L^1_{\text{loc}}(\R_+\times \R)$, which is the case by \eqref{alphanorm}. If $\alpha<1$, we check that for any compact $K\subset \R^2$,
$$\int_{\R_+\times \R} \dd s \dd y  \lp \int_K \dd t \dd x\,  |\rho_1^\sO(t-s,x-y) |  \rp ^\alpha <+\infty\, .$$
It is easy to see that the function $(s,y) \mapsto \int_K \dd t \dd x\,  |\rho_1^\sO(t-s,x-y) | $ has compact support, which suffices to prove the claim. In the case $\alpha=1$, we check that for any compact set $K\subset \R^2$,
\begin{equation}\label{logmoche}
\begin{aligned}
  \int_{K}& \dd t \dd x  \int_{\R_+\times \R} \! \dd s \dd y \, |  \rho^\sO_1(t-s,x-y)| \left [ 1+ \vphantom{\log_+ \lp \frac{\int_1^2 | \rho^\sO_1(t-s,x-y)|}{\int_1^2}\rp} \right. \\
 &  \left. \log_+ \lp \frac{| \rho^\sO_1(t-s,x-y)| \int_{K}\dd u \dd r \int_{\R_+\times \R} \dd v \dd w | \rho^\sO_1(u-v,r-w)|     }{\lp\int_{\R_+\times \R} \dd v \dd w | \rho^\sO_1(t-v,x-w)| \rp\lp \int_{K} | \rho^\sO_1( u-s,r-y) |  \dd u \dd r \rp} \rp \right ] <+\infty \, .
\end{aligned}
\end{equation}
The details of this calculation can be found in \cite[Proof of Proposition 4.4.9]{humeau1}.

We conclude that, for any $\alpha \in (0,2)$, the mild solution is equal to the generalized solution, and the mild solution is the random field representation of the generalized solution.
\vskip 12pt

\noindent
\underline{$d=2$:}
In the case where $\alpha>1$, by \eqref{swealphanorm}, $(t,x)\in \R_+\times \R^2 \to \| \rho_2^\sO (t-\cdot, x-\cdot) \|_{L^\alpha(\R_+\times \R^2)}$ does not depend on $x$ and is continuous in the $t$ variable, therefore \eqref{alpha0} is verified, and the mild solution is equal to the generalized solution.

In the case where $\alpha<1$, we know from previous considerations that for any test function $\varphi$, $| \check  \rho^\sO_2 | *\varphi $ is smooth with compact support, therefore \eqref{alphag1} is verified, and the mild solution is equal to the generalized solution.

The case $\alpha=1$ is again more involved, since we need to consider the expression \eqref{alpha1}. We must check that for any compact set $K\subset \R^3$,
\begin{equation}\label{logmoche3}
\begin{aligned}
  \int_{K}& \dd t \dd x  \int_{\R_+\times \R^2} \! \dd s \dd y \, |  \rho^\sO_2(t-s,x-y)| \left [ 1+ \vphantom{\log_+ \lp \frac{\int_1^2 | \rho^\sO_2(t-s,x-y)|}{\int_1^2}\rp} \right. \\
 &  \left. \log_+ \lp \frac{| \rho^\sO_2(t-s,x-y)| \int_{K}\dd u \dd r \int_{\R_+\times \R^2} \dd v \dd w | \rho^\sO_2(u-v,r-w)|     }{\lp\int_{\R_+\times \R^2} \dd v \dd w | \rho^\sO_2(t-v,x-w)| \rp\lp \int_{K} | \rho^\sO_2( u-s,r-y) |  \dd u \dd r \rp} \rp \right ] <+\infty \, .
\end{aligned}
\end{equation}
For the details of this calculation, see \cite[Proof of 4.4.9]{humeau1}.
Therefore, for any $\alpha \in (0,2)$, the mild solution is equal to the generalized solution. and the mild solution is the random field representation of the generalized solution.
\vskip 12pt

\noindent
\underline{$d\s 3$:}
By Theorem \ref{rfieldsolution}, there cannot be any random field representation of the generalized solution, since $\rho^\sO_d \notin L^\alpha \lp [0,T]\times \R^d\rp$.
\end{proof}

We summarize Propositions \ref{rdprop6.7}, \ref{rdprop6.8} and \ref{rdprop6.10} in the following theorem.

\begin{theo}\label{wave_result}
 The generalized solution $u_{\text{gen}}$ to the stochastic wave equation \eqref{SWE} defined by \eqref{def1} always exists. The mild solution $u_{\text{mild}}$ defined by \eqref{solutionmildheat} exists if and only if $d\I 2$.
Furthermore, a random field representation $Y$ of the generalized solution exists if and only if $d\I 2$, and in this case, for almost all $(t,x) \in \R_+\times \R^d$,
\begin{equation*}
 Y_{t,x} =\scal{\dot W^\alpha}{\rho^\sO_d(t-\cdot, x-\cdot)}= u_{\text{mild}}(t,x) \qquad \text{a.s.} 
\end{equation*}
\end{theo}

\subsection{The stochastic Poisson equation}
Let $\nalpha W$ be an $\alpha$-stable symmetric noise on $\R_+\times \R^d$. The laplacian operator $\Delta$ is given by
$$
   \Delta = \sum_{i=1}^d \frac{\partial^2}{\partial x_i^2}\, .
$$
The fundamental solution of the Poisson operator $\sP = - \Delta$ on $\R^d$ is given by
\begin{align*}
   \rho_\sP^1(x) &= \frac{1}{2} \vert x \vert, \qquad\quad x \in \R,\\
   \rho_\sP^2(x) &= \frac{1}{2\pi} \ln\frac{1}{\vert x \vert} , \qquad x \in \R^2 \setminus\{0\},\\
   \rho_\sP^d(x) &= \frac{1}{C_d} \frac{1}{\vert x \vert^{d-2}},\qquad x \in \R^d \setminus\{0\},\ d \geq 3,
\end{align*}
where
$$
   C_d = \frac{2 \pi^{\frac{d}{2}}(d-2)}{\Gamma(\frac{d}{2})}.
$$

   We consider the following SPDE in $\R^d$:
\begin{equation}\label{poisson_eq}
   - \Delta u = \nalpha W .
\end{equation}

\begin{theo}\label{poisson_result}
(a) For $d \geq 1$ and $\alpha \in (0,2)$, there is no mild solution, hence no random field solution, to \eqref{poisson_eq}.

(b) There is a generalized solution to \eqref{poisson_eq} if and only if $d>4$ and $\alpha \in \left(\frac{d}{d-2},2\right)$.
\end{theo}

\begin{proof}
(a) It is immediate to check that for all $d\geq 1$ and $\alpha \in (0,2)$, $\rho_\sP^d \notin L^\alpha(\R^d)$, therefore, there is no mild solution to \eqref{poisson_eq}.

(b) Turning to the generalized solution, we first examine dimensions $1$ and $2$.
\vskip 12pt

\noindent
\underline{$d=1$:} Let $\varphi \in \m D(\R)$. Then
\begin{align*}
   \check \rho_\sP^1 * \varphi(x) &= -\frac{1}{2} \int_\R \vert x-y\vert \varphi(y)\, dy \\
   &= -\frac{1}{2} \left(x \left(\int_{-\infty}^x \varphi(y)\, dy - \int_x^{+\infty} \varphi(y) \, dy \right) + \int_x^{+\infty} y \varphi(y) \, dy - \int_{-\infty}^x y \varphi(y)\, dy \right).
\end{align*}
Since $\varphi$ has compact support, for large enough $\vert x \vert$,
\begin{equation}\label{6.15}
   \check \rho_\sP^1 * \varphi(x) = -\frac{1}{2} \left(x \int_\R  \varphi(y)\, dy - \int_\R  y \varphi(y)\, dy \right).
\end{equation}
In particular, we see that for any $\alpha \in (0,2)$, $\check \rho_\sP^d * \varphi \notin L^\alpha(\R)$ (unless the two integrals in \eqref{6.15} vanish), and there is no generalized solution to \eqref{poisson_eq}.
\vskip 12pt

\noindent
\underline{$d=2$:} Let $\varphi \in \m D(\R^{2})$. Then
\begin{align*}
   \check \rho_\sP^2 * \varphi(x) &= \frac{1}{2\pi} \int_{\R^2} \ln \frac{1}{\vert x-y\vert}\,  \varphi(y)\, dy.
\end{align*}
Assuming that $\varphi \geq 0$ and $\varphi \not\equiv 0$, for $\vert x \vert$ large enough and $\varepsilon$ small enough, the right-hand side is bounded below by
$$
   \frac{\Vert \varphi\Vert_\infty}{2} \int_{\vert y \vert \leq \varepsilon}  \ln \frac{1}{\vert x-y\vert}\, dy \geq \varepsilon^2 \frac{\Vert \varphi\Vert_\infty}{2} \frac{1}{\ln(\vert x \vert +\varepsilon)},
$$
hence $\check \rho_\sP^2 * \varphi \notin L^\alpha(\R^2)$.
\vskip 12pt

\noindent
\underline{$d \geq 3$:} Let $\varphi \in \m D(\R^{d})$. Then
\begin{align*}
   \check \rho_\sP^d * \varphi(x) &= \int_{\R^d} \frac{1}{\vert x-y\vert^{d-2}} \, \varphi(y) \, dy.
\end{align*}
This is a $C^\infty$-function of $x$, hence we only need to consider its integrability as $\vert x\vert \to \infty$. Proceeding as for the case $d=2$, for $\varphi \geq 0$ and $\varphi \not\equiv 0$,  we can bound the integral below, up to a constant, by
$$
   \int_{\vert y \vert \leq \varepsilon} \frac{1}{\vert x-y\vert^{d-2}} \, dy \geq \varepsilon^d \frac{1}{(\vert x \vert +\varepsilon)^{d-2}}.
$$
Passing to polar coordinates, we see that $\check \rho_\sP^d * \varphi \notin L^\alpha(\R^d)$ unless $\alpha (2-d) + d <0$, which is equivalent to $\alpha > \frac{d}{d-2}$. Since $\alpha < 2$, this can only occur if $d >4$.

On the other hand, if $d>4$ and $\alpha > \frac{d}{d-2}$, then for $N >0$, by the generalized Minkowski inequality,
$$
   \int_{\vert x \vert > N} dx \left \vert \int_{\R^d} \frac{1}{\vert x-y\vert^{d-2}} \, \varphi(y) \, dy \right\vert^\alpha \leq \int_{\R^d} dy\, \vert \varphi(y)\vert^\alpha \left \vert \int_{\vert x \vert > N} \frac{1}{\vert x-y\vert^{\alpha(d-2)}}  \, dx \right\vert
$$
The $dx$-integral only needs to be evaluated for $y$ in a bounded set. For large enough $N$, the $dx$-integral is finite if and only if $\alpha (2-d) + d <0$, which is the case since $d>4$ and $\alpha > \frac{d}{d-2}$.

In summary, $\check \rho_\sP^d * \varphi \in L^\alpha(\R^2)$ if and only if $d>4$ and $\alpha > \frac{d}{d-2}$, and statement (b) is proved.
\end{proof}


\begin{thebibliography}{10}

\bibitem{balan}
Balan, R.~M.
\newblock S{PDE}s with {$\alpha$}-stable {L}\'evy noise: a random field
  approach.
\newblock {\em Int. J. Stoch. Anal.}, pp. Art. ID 793275, 22, 2014.

\bibitem{basse}
Barndorff-Nielsen, O.~E. and Basse-O'Connor, A.
\newblock Quasi {O}rnstein-{U}hlenbeck processes.
\newblock {\em Bernoulli}, 17(3):916--941, 2011.

\bibitem{chongheavytailed}
Chong, C.
\newblock Stochastic {PDE}s with heavy-tailed noise.
\newblock {\em Stochastic Process. Appl.}, 127(7):2262--2280, 2017.

\bibitem{chong_integrability}
Chong, C. and Kl{\"u}ppelberg, C.
\newblock Integrability conditions for space-time stochastic integrals: theory
  and applications.
\newblock {\em Bernoulli}, 21(4):2190--2216, 2015.

\bibitem{conusphd}
Conus, D.
\newblock {\em The Non-linear Stochastic Wave Equation in High Dimensions:
  Existence, H{\"o}lder-continuity and It{\^o}-Taylor Expansion}.
\newblock PhD thesis, EPFL, 2008.

\bibitem{minicourse}
Dalang, R., Khoshnevisan, D., Mueller, C., Nualart, D., and Xiao, Y.
\newblock {\em A minicourse on stochastic partial differential equations},
  volume 1962 of {\em Lecture Notes in Mathematics}.
\newblock Springer-Verlag, Berlin, 2009.
\newblock Held at the University of Utah, Salt Lake City, UT, May 8--19, 2006,
  Edited by Khoshnevisan and Firas Rassoul-Agha.

\bibitem{dalang99}
Dalang, R.~C.
\newblock Extending the martingale measure stochastic integral with
  applications to spatially homogeneous s.p.d.e.'s.
\newblock {\em Electron. J. Probab.}, 4:no.\ 6, 29 pp.\ (electronic), 1999.

\bibitem{fageot_humeau}
{Fageot}, J. and {Humeau}, T.
\newblock {Unified View on L\'evy White Noises: General Integrability
  Conditions and Applications to Linear SPDE}.
\newblock {\em ArXiv e-prints}, (1708.02500), August 2017.

\bibitem{gelfand}
Gel'fand, I.~M. and Vilenkin, N.~Y.
\newblock {\em Generalized functions. {V}ol. 4}.
\newblock Academic Press [Harcourt Brace Jovanovich, Publishers], New
  York-London, 1964 [1977].

\bibitem{distributions}
Golse, F.
\newblock {\em Distributions, Analyse de Fourier, {\'e}quations aux
  d{\'e}riv{\'e}es partielles}.
\newblock Ecole Polytechnique, 2010.

\bibitem{humeau1}
Humeau, Th. {\em Stochastic partial differential equations driven by L\'evy white noises: Generalized random processes, random field solutions and regularity.} PhD.~Thesis no.8223, \'Ecole Polytechnique F\'ed\'erale de Lausanne, Switzerland (2017).

\bibitem{kallenberg}
Kallenberg, O.
\newblock {\em Foundations of modern probability}.
\newblock Springer-Verlag, New York, 2nd edition, 2002.

\bibitem{davar}
Khoshnevisan, D.
\newblock {\em Analysis of stochastic partial differential equations}, volume
  119 of {\em CBMS Regional Conference Series in Mathematics}.
\newblock American Mathematical Society, Providence, RI, 2014.

\bibitem{lebedev}
Lebedev, V.~A.
\newblock Fubini's theorem for parameter-dependent stochastic integrals with
  respect to {$L^0$}-valued random measures.
\newblock {\em Teor. Veroyatnost. i Primenen.}, 40(2):313--323, 1995.

\bibitem{PZ}
Peszat, S. and Zabczyk, J. {\em Stochastic partial differential equations with Lévy noise. An evolution equation approach.} Encyclopedia of Mathematics and its Applications, 113. Cambridge University Press, Cambridge, 2007.

\bibitem{rosinski}
Rajput, B.~S. and Rosi{\'n}ski, J.
\newblock Spectral representations of infinitely divisible processes.
\newblock {\em Probab. Theory Related Fields}, 82(3):451--487, 1989.

\bibitem{path}
Rosi{\'n}ski, J.
\newblock On path properties of certain infinitely divisible processes.
\newblock {\em Stochastic Process. Appl.}, 33(1):73--87, 1989.

\bibitem{taqqu}
Samorodnitsky, G. and Taqqu, M.~S.
\newblock {\em Stable non-gaussian random processes: Stochastic models with
  infinite variance}.
\newblock Stochastic Modeling. Chapman \& Hall, 1994.

\bibitem{distributions_schwartz}
Schwartz, L.
\newblock {\em Th\'eorie des distributions}.
\newblock Publications de l'Institut de Math\'ematique de l'Universit\'e de
  Strasbourg, No. IX-X. Hermann, Paris, 1966.

\bibitem{stein}
Stein, E.~M.
\newblock {\em Singular integrals and differentiability properties of
  functions}.
\newblock Princeton Mathematical Series, No. 30. Princeton University Press,
  Princeton, N.J., 1970.

\bibitem{treves}
Tr{\`e}ves, F.
\newblock {\em Topological vector spaces, distributions and kernels}.
\newblock Academic Press, New York-London, 1967.

\bibitem{spdewalsh}
Walsh, J.~B.
\newblock An introduction to stochastic partial differential equations.
\newblock In {\em \'{E}cole d'\'et\'e de probabilit\'es de {S}aint-{F}lour,
  {XIV}---1984}, volume 1180 of {\em Lecture Notes in Math.}, pp. 265--439.
  Springer, Berlin, 1986.

\bibitem{zygmund}
Wheeden, R.~L. and Zygmund, A.
\newblock {\em Measure and integral}.
\newblock Marcel Dekker, Inc., New York-Basel, 1977.
\newblock An introduction to real analysis, Pure and Applied Mathematics, Vol.
  43.

\end{thebibliography}

\end{document}